\documentclass[11pt,a4paper,reqno]{amsart}
\usepackage{amsmath,amsthm,amsfonts,amssymb,bm,wasysym}
\usepackage{comment}
\usepackage{yhmath}
\usepackage{caption}
\usepackage{subcaption}
\usepackage{fullpage}
\usepackage{graphicx}
\usepackage{float}
\usepackage{enumitem}
\usepackage{array}
\usepackage{tikz}
\usepackage{hyperref}
\usetikzlibrary{calc}
\usetikzlibrary{patterns}
\usetikzlibrary{arrows}
\usetikzlibrary{decorations.pathreplacing}
\theoremstyle{plain}
\newtheorem{theorem}{Theorem}
\newtheorem{proposition}[theorem]{Proposition}
\newtheorem{corollary}[theorem]{Corollary}
\newtheorem{lemma}[theorem]{Lemma}

\theoremstyle{definition}
\newtheorem{definition}[theorem]{Definition}

\theoremstyle{remark}
\newtheorem{remark}{Remark}

\begin{document}
\title{On the absence of percolation in a line-segment based lilypond model}
\def\A{\mathbb{A}}

\def\Ab{\mathcal{A}b}

\def\absq{{a^{\prime}}^2+{b^\prime}^2}

\def\AP{\text{G}}

\def\app{{a^{\prime\prime}}^2+1}

\def\argmin{\text{argmin}}

\def\arb{arbitrary }

\def\ass{assumption}

\def\arrow{\rightarrow}

\def\codim{\text{codim}}

\def\const{c}

\def\CCG{\text{G}}

\def\colim{\text{colim}}

\def\cond{condition }

\def\C{\mbox{\bf C}}

\def\d{{\rm d}}

\def\dell{\partial}

\def\diam{\text{diam}}

\def\E{\mathbb{E}}

\def\envi{\mathsf{env}}

\def\enviIn{\partial^{\mathsf{in}}}

\def\enviOut{\partial^{\mathsf{out}}}

\def\enviInn{\partial^{\mathsf{in},*}}

\def\enviStab{\mathsf{env}_{\mathsf{stab}}}

\def\Et{\text{Et}}

\def\es{\emptyset}

\def\exp{\text{exp}}

\def\fa{for all }

\def\Fk{\mathcal{F}_{k_0}}

\def\Fm{Furthermore}

\def\G{\mathbb{G}}

\def\gr{\text{gr}}

\def\hge{h_{\mathsf{g}}}

\def\hco{h_{\mathsf{c}}}

\def\H{\text{H}}

\def\Hom{\text{Hom}}

\def\Hs{\widetilde{X}_{H,0}}

\def\inj{\hookrightarrow}

\def\id{\text{id}}

\def\iiets{it is easy to see }

\def\iietc{it is easy to check }

\def\Iietc{It is easy to check }

\def\Iiets{It is easy to see }

\def\imp{\Rightarrow}

\def\({\left(}

\def\){\right)}

\def\[{\left[}

\def\]{\right]}

\def\lver{\left\lvert}

\def\rver{\right\rvert}

\def\lcu{\left\{}

\def\rcu{\right\}}

\def\im{\mbox{im}}

\def\Inv{\text{Inv}}

\def\Ind{\text{Ind}}

\def\Ip{In particular}

\def\ip{in particular }

\def\LB{\text{LB}}

\def\Lo{\mathcal{L}^o}

\def\mc{\mathcal}

\def\mb{\mathbb}

\def\mf{\mathbf}

\def\Hp{\wt{X}_{H,0}^{'}}

\def\M{\mathbb{M}}

\def\Mo{Moreover}

\def\G{\mathbb{G}}

\def\N{\mathbb{N}}

\def\Npo{\mathbf{N}_{\mathcal{P}^o}}

\def\k{\overline{k}}

\def\K{\underline{K}}

\def\ldot{.}

\def\O{\mathcal{O}}

\def\Ob{Observe }

\def\ob{observe }

\def\Otoh{On the other hand}

\def\opartial{\partial^{\text{out}}}

\def\ipartial{\partial^{\text{in}}}

\def\eopartial{\partial^{\text{out}}_{\text{ext}}}

\def\eipartial{\partial^{\text{in}}_{\text{ext}}}

\def\p{\prime}

\def\pp{{\prime\prime}}

\def\Po{\mathcal{P}^0}

\def\P{\mathbb{P}}

\def\Proj{\mbox{\bf P}}

\def\Q{\mathbb{Q}}

\def\QQ{\overline{\Q}}

\def\pr{\text{pr}}

\def\R{\mathbb{R}}

\def\rstab{R_{\mathsf{stab}}}

\def\Spec{\text{Spec}}

\def\st{such that }

\def\sl{sufficiently large }

\def\ss{sufficiently small }

\def\sot{so that }

\def\su{suppose }

\def\Su{Suppose }

\def\suf{sufficiently }

\def\udot{\mathaccent\cdot\cup}

\def\Set{\mathcal{S}et}

\def\Twh{Then we have }

\def\Tes{There exists }

\def\te{there exist }

\def\tes{there exists }

\def\tptc{this proves the claim}

\def\Map{\text{Map}}

\def\VLo{\mc{VL}^o}

\def\wt{\widetilde}

\def\Wcon{We conclude }

\def\wcon{we conclude }

\def\wc{we compute }

\def\Wc{We compute }
\def\wo{we obtain }
\def\wh{we have }
\def\Wh{We have }
\def\Z{\mathbb{Z}}
\def\ZSlab{\mathbb{Z}^2_L\times\{0\}^{d-2}}


\author{Christian Hirsch}
\thanks{Institute of Stochastics, Ulm University, 89069 Ulm, Germany; E-mail: {\tt christian.hirsch@uni-ulm.de}.}

\begin{abstract}
We prove the absence of percolation in a directed Poisson-based random geometric graph with out-degree $1$. This graph is an anisotropic variant of a line-segment based lilypond model obtained from an asymmetric growth protocol, which has been proposed by Daley and Last. In order to exclude backward percolation, one may proceed as in the lilypond model of growing disks and apply the mass-transport principle. Concerning the proof of the absence of forward percolation, we present a novel argument that is based on the method of sprinkling.
\end{abstract}


\maketitle

\section{Introduction}
\label{intSec}
The classical lilypond model describes a hard-sphere system defined by the following growth-stopping protocol. Start with a planar homogeneous Poisson point process $X$ whose atoms serve as germs of a growth process. At time $0$ and with the same speed at each element $x\in X$ a spherical grain begins to grow. As soon as one such grain touches another both cease to grow. This model and various generalizations have been intensively studied for almost $20$ years, \sot today an entire family of results concerning existence, uniqueness, stabilization and absence of percolation is known. We refer the reader to the original articles \cite{lilypond,lilypond2,lilypond5,nnhs,lilypond3,lilypond4} for details. The purpose of the present paper is to further advance the completion of this picture by adding a result on the absence of percolation in a lilypond model based on an asymmetric growth-stopping protocol. To be more precise, we consider a model where from each atom of a planar Poisson point process a line segment starts to grow in one of the directions $\pm e_1=(\pm1,0)$, $\pm e_2=(0,\pm1)$ and as soon as a line segment touches an already existing one, the former ceases to grow. This model is an anisotropic variant of one of the two line-segment based lilypond models introduced in~\cite{lilypond6}. A realization of the anisotropic lilypond model is shown in Figure~\ref{sfbmFig}.

The lilypond model gives rise to a directed graph on $X$, where an edge is drawn from $x$ to $y$ if the growth of the line segment at $x$ is stopped by the line segment at $y$. We prove that with probability $1$, this graph exhibits neither forward nor backward percolation, thus verifying~\cite[Conjecture 7.1]{lilypond6} in an anisotropic, one-sided setting. The investigation of the absence of percolation in lilypond-type models has been initiated in~\cite{nnhs} and we briefly review the main idea to establish the absence of percolation in models using spherical grains. After that, we explain which parts of the proof have to be modified in the line-segment setting.

In lilypond models based on spherical grains, the notion of doublets plays a crucial role, where a \emph{doublet} consists of a pair of disk-shaped grains $B_1,B_2\subset\R^2$ \st $B_1$ stops the growth of $B_2$ and $B_2$ stops the growth of $B_1$. This notion allows to subdivide the proof for the absence of percolation provided in~\cite{lilypond} into two steps. In the first step it is shown that a.s. each connected component of the lilypond model contains at most one doublet. In the second step, the a.s. absence of descending chains for homogeneous Poisson point processes is used to show that every connected component also contains at least one doublet. \Ip, by mapping each connected component to the midpoint of the two doublet centers we are able construct a locally finite set from the family of connected components in a translation-covariant way. Therefore, an application of the mass-transport principle implies the absence of infinite connected components.

\begin{figure}[!htpb]
\centering
 {\includegraphics[width=7.0cm]{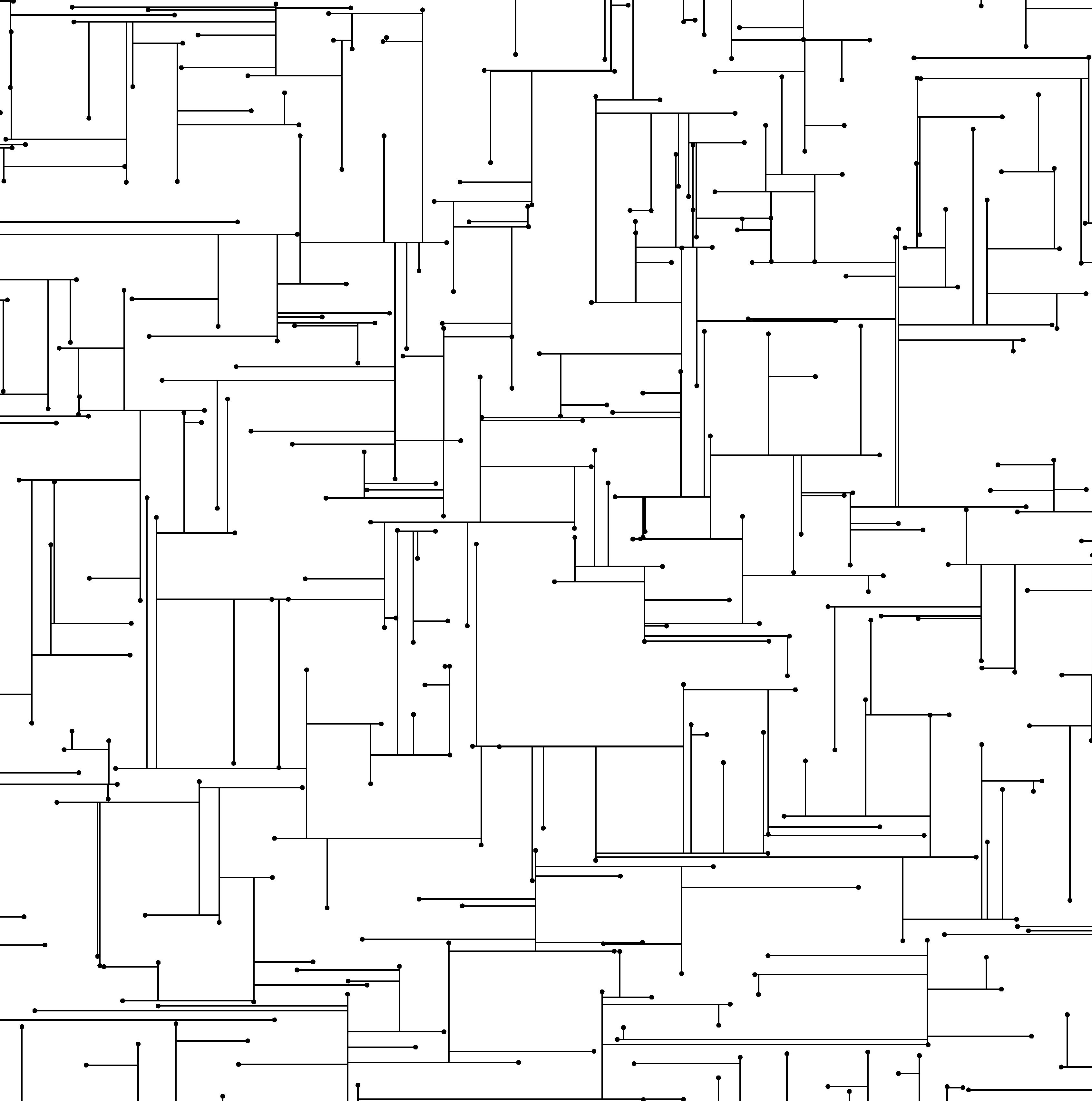}}
\caption{Realization of the lilypond line-segment model (cutout)}\label{sfbmFig}
\end{figure}

In our setting, the notion of doublets is replaced by cycles, where a \emph{cycle} consists of a sequence of line segments $L_1,L_2,\ldots,L_k\subset\R^2$ \st  each $L_{i+1}$ ($i=1,\ldots,k-1$) stops the growth of $L_i$ and $L_1$ stops the growth of $L_k$. As in the setting of spherical grains, it is clear that any connected component contains at most one cycle. \Fm, another application of the mass-transport principle proves the absence of infinite connected components containing a cycle. \Otoh, to show that every connected component contains at least one cycle, we use the \emph{sprinkling} technique developed in~\cite{sprinkling}. In other words, we first express the planar homogeneous Poisson point process $X$ as superposition of two independent homogeneous Poisson point processes $X=X^{(1)}\cup X^{(2)}$, where the intensity of $X^{(1)}$ is only slightly smaller than the intensity of $X$. When considering the lilypond model on $X^{(1)}$, this graph could contain connected components without a cycle, a priori. The idea of the proof is to show that sprinkling the remaining centers $X^{(2)}$ has the effect of stopping every semi-infinite directed path in the lilypond model based on the point process $X^{(1)}$ and that if the sprinkling intensity is chosen sufficiently small, then no additional semi-infinite paths appear. One key step in the formalization of this idea is to combine a stabilization result for the lilypond model at hand with a standard result on dependent percolation~\cite{domProd} to ensure that, except for small exceptional islands in the plane, one has good control on the effects of the sprinkling. 

Our method uses only rather general properties of the line-segment based lilypond model and it might be useful to prove the absence of percolation in further directed Poisson-based random geometric graphs with out-degree $1$. Indeed, the sprinkling technique applies if the underlying graph satisfies a suitable shielding condition and there is a positive probability of modifying the graph locally inside large square so that any path entering the square is stopped. 

The paper is organized as follows. In Section~\ref{defSec}, we provide a precise description of the lilypond model under consideration and state the main result of this paper, Theorem~\ref{mainProp}, which deals with the absence of percolation. In Section~\ref{bPercSec}, we explain how the mass-transport principle can be used to deduce the absence of backward percolation from the absence of forward percolation, and we also state several auxiliary results, which are important in the proof of the absence of forward percolation. Assuming these auxiliary results, in Section~\ref{ideaSec}, we prove the absence of forward percolation using the sprinkling technique. Section~\ref{lilySec} is devoted to the proof of the auxiliary results. Finally, in Section~\ref{extSec}, we discuss possible extensions of the sprinkling technique to other directed Poisson-based random geometric graphs of outdegree at most $1$.

\section{Model definition and statement of main result}
\label{defSec}
The purpose of this section is two-fold. First, we provide a formal definition of the line-segment based asymmetric lilypond model which shall be the topic of our considerations. Second, we state the main result of the present paper, Theorem~\ref{mainProp}.

In the lilypond model under consideration, at time $0$ from every point of an independently marked homogeneous planar Poisson point process $X$ with intensity $1$, a line segment starts to grow in one of the four directions $\M=\lcu\pm e_1,\pm e_2\rcu$ (which is chosen uniformly at random). It stops growing as soon as it hits another line segment. The growth-stopping protocol is asymmetric in the sense that in contrast to the hitting line segment, the segment being hit does \emph{not} stop growing (provided of course that its growth had not already stopped before the collision). Although the growth dynamics of this lilypond model admits a very intuitive description, providing a rigorous mathematical definition is not entirely trivial. Nevertheless, by now this problem has been investigated for many variants of the classical lilypond model from~\cite{nnhs} and existence as well as uniqueness are guaranteed if the underlying point process does not admit a suitable form of descending chains. These chains are usually easy to exclude for independently marked homogeneous Poisson point processes, see e.g.~\cite{lilypond6,lilypond,lilypond5}. For our purposes the correct variant is the following. 
\begin{definition}
\label{aniDescChainDef}
Let $b>0$ and $\varphi\subset\R^{2}$ be locally finite. A (finite or infinite) set $\{\xi_i\}_{i\ge1}\subset\varphi$ is said to form a \emph{$b$-bounded anisotropic descending chain} if
 $\left|\xi_1-\xi_2\right|_\infty\le b$ and
 $\left|\xi_i-\xi_{i+1}\right|_\infty<\left|\xi_{i-1}-\xi_i\right|_\infty$ \fa $i\ge2$.
A set $\{\xi_i\}_{i\ge1}\subset\varphi$ is said to define an \emph{anisotropic descending chain} if it forms a \emph{$b$-bounded anisotropic descending chain} for some $b>0$.
\end{definition}
\Ip, one can derive the following result whose proof is obtained by a straightforward adaptation of the arguments in~\cite{lilypond6}, where we write $\N_\M$ for the family of all locally finite subsets of $\R^{2,\M}=\R^2\times\M$. 
\begin{proposition}
\label{lilyPreDef}
Let $\varphi\in\N_\M$ be an $\M$-marked locally finite set that does not contain anisotropic descending chains and \st $(\xi-\eta)/\lver \xi-\eta \rver\not\in\lcu\pm e_1,\pm e_2,(\pm e_1\pm e_2)/\sqrt{2}\rcu $ \fa $x=(\xi,v),y=(\eta,w)\in\varphi$ with $x\ne y$. Then, \tes a unique function $f:\varphi\to [0,\infty]$ with the following properties.
\begin{enumerate}
\item $[\xi,\xi+f(x)v)\cap [\eta,\eta+f(y)w)=\es$ \fa $x=(\xi,v),y=(\eta,w)\in\varphi$ with $x\ne y$ (hard-core property), and 
\item for every $x\in\varphi$ with $f(x)<\infty$ \tes a unique $y=(\eta,w)\in\varphi$ \st $\xi+f(x)v\in [\eta,\eta+f(y)w)$ and $\lver \xi+f(x)v-\eta\rver<f(x)$ (existence of stopping neighbors).
\end{enumerate}
\end{proposition}

In the following, we denote by $\N^\p$ the family of all non-empty $\varphi\in\N_\M$ \st $f(x)<\infty$ \fa $x\in\varphi$, \st $(\xi-\eta)/\lver \xi-\eta \rver\not\in\lcu\pm e_1,\pm e_2,(\pm e_1\pm e_2)/\sqrt{2}\rcu $ \fa $x=(\xi,v),y=(\eta,w)\in\varphi$ with $x\ne y$, and \st $\varphi$ does not contain anisotropic descending chains. \Fm, it will be convenient to introduce functions $\hco:\N^\p\times \R^{2,\M}\to \R^{2,\M}$ and $\hge:\N^\p\times\R^{2,\M}\to\R^{2}$, where $\hco(\varphi,x)$ denotes the uniquely determined stopping neighbor of $x$ (in the sense of point $2.$ in Proposition~\ref{lilyPreDef}), and where $\hge(\varphi,(\xi,v))=\xi+f(x)v$. In other words, $\hco(\varphi,x)$ denotes the element of $\varphi$ stopping the growth of $x$, whereas $\hge(\varphi,x)$ denotes the actual endpoint of the segment emanating from $x$.  Therefore, we call $\hco(\varphi,x)$ the \emph{combinatorial descendant} and $\hge(\varphi,x)$ the \emph{geometric descendant} of $x$. If $x=(\xi,v)\not\in\varphi$, we put $\hco(\varphi,x)=x$ and $\hge (\varphi,x)=\xi$. 

Our results on the absence of percolation can be stated using only the notion of combinatorial descendants. Still, tracing the path described by following iteratively the geometric descendants is in a sense much closer to the geometry of the underlying lilypond model than tracing the path of iterated combinatorial descendants. Thus, it is not surprising that geometric descendants play a crucial role in the analysis of percolation properties. This justifies the introduction of separate notation despite the fact that $\hge(\varphi,x)$ could be easily recovered from $x$, $\varphi$ and $\hco(\varphi,x)$. 

In order to state our main result, it is convenient to introduce for any $x\in X$ the set $\hco^{(\infty)}(X,x)=\big\{ \hco^{(n)}(X,x):n\ge0\big\}$, where we recursively define $\hco^{(0)}(X,x)=x$ and $\hco^{(n)}(X,x)=\hco\big(X,\hco^{(n-1)}(X,x)\big)$, $n\ge1$.
\begin{theorem}
\label{mainProp}
With probability $1$, the lilypond line-segment model does not percolate, i.e.,
\begin{enumerate}
\item for every $x\in X$ the set $\hco^{(\infty)}(X,x)$ is finite, and\label{part1}
\item for every $x\in X$ there exist only finitely many $y\in X$ with $x\in \hco^{(\infty)}(X,y)$.\label{part2}
\end{enumerate}
\end{theorem}

\section{Absence of backward percolation and statement auxiliary results}
\label{bPercSec}
The goal of the present section is two-fold. First, we show how the absence of backward percolation (part 2. of Theorem~\ref{mainProp}) can be derived from the absence of forward percolation (part 1. of Theorem~\ref{mainProp}) using the mass-transport principle. Second, we highlight three important properties of the lilypond model, which will be verified in Section~\ref{lilySec}. The benefit of introducing these properties in the present section is that these are the main properties of the lilypond model that will be used in the proof for the absence of forward percolation in Section~\ref{ideaSec}. In this way, we separate the presentation of the sprinkling method from the rather technical verification of the three properties.

In the following, for $r>0$ and $\xi\in\R^2$, we denote by $Q_r(\xi)=[-r/2,r/2]^2+\xi$ the square of side length $r$ centered at $\xi$. We also put $Q_r^\M(\xi)=Q_r(\xi)\times\M$.
To begin with, we deduce the absence of backward percolation from the absence of forward percolation.
\begin{proof}[Proof of Theorem~{\hyperref[part2]{\ref*{mainProp}.\ref*{part2}}} assuming Theorem~{\hyperref[part2]{\ref*{mainProp}.\ref*{part1}}}]
Similar to the arguments used in~\cite{lilypond,lilypond5}, we use the mass-transport principle. Loosely speaking, to define a translation-covariant mass transport, we first note that from the absence of forward percolation, we deduce that starting from any point of the Poisson point process and taking iterated combinatorial descendants we arrive at a cycle. Transporting one unit of mass from that point to the center of gravity of the cycle, we see that choosing a discretization of the Euclidean space into squares, the expected total outgoing mass from any square is finite, whereas the occurrence of backward percolation would result in some square receiving an infinite amount of mass.

To be more precise, for every $x\in X$ we denote by $V(x)$ the set of all $y\in X$ \st $\hco^{(n)}(\varphi,x)=y$ for infinitely many $n\ge1$ and by $C(x)$ the center of gravity of the spatial coordinates in $V(x)$. Since Theorem~{\hyperref[part2]{\ref*{mainProp}.\ref*{part1}}} implies that $V(x)$ is finite, this point is well-defined. Next, we introduce a function $\psi:\Z^2\times\Z^2\to [0,\infty)$ by putting
$$\psi(z_1,z_2)=\#\{x\in X\cap Q^\M_1(z_1):C(x)\in Q_1(z_2)\},$$
\sot $\psi(z_1,z_2)$ denotes the number of elements of $x\in X\cap Q^\M_1(z_1)$ \st $C(x)$ is contained in $Q_1(z_2)$. Note that if $x\in X$ is \st $C(x)\in Q_1(o)$ and there exist infinitely $y\in X$ with $x\in\hco^{(\infty)}(X,y)$, then $\sum_{z\in\Z^2}\psi(z,o)=\infty$. \Fm, for any $z\in\Z^d$ the random variables $\psi(z,o)$ and $\psi(o,-z)$ have the same distribution, \sot the assumption from Section~\ref{defSec} that $X$ is a homogeneous Poisson point process with intensity $1$ yields
\begin{align*}
\E\sum_{z\in\Z^2}\psi(z,o)=\sum_{z\in\Z^2}\E\psi(z,o)=\sum_{z\in\Z^2}\E\psi(o,-z)=\E\sum_{z\in\Z^2}\psi(o,-z)=\E\#(X\cap Q^\M_1(o))=1.
\end{align*}
\Ip, $\sum_{z\in\Z^2}\psi(z,o)$ is a.s. finite, \sot with probability $1$ there does not exist $x\in X$ \st $C(x)\in Q_1(o)$ and \st $x\in\hco^{(\infty)}(X,y)$ for infinitely many $y\in X$. Using stationarity once more completes the proof of Theorem~{\hyperref[part2]{\ref*{mainProp}.\ref*{part2}}}.
\end{proof}

In the proof of the absence of forward percolation the sprinkling method~\cite{sprinkling} is used. In the first step, we form the graph based on all but a tiny fraction of $X$, while in the second step the remaining points of $X$ are added independently in order to stop any of the possibly existing infinite paths.
In order to turn this rough description into a rigorous proof, we make use of three important properties of the lilypond model. In the present section, we state these properties and provide explanations and illustrations in order to make the reader familiar with them. Next, in Section~\ref{ideaSec}, we provide a proof for the absence of forward percolation based on these properties. Finally, in Section~\ref{lilySec}, we verify these properties for the specific lilypond model under consideration. We present the three properties in order of increasing complexity.

First, we note that the combinatorial descendant function $\hco$ satisfies a continuity property in the sense that if $\varphi\in\N^\p$ and $\varphi_1\subset\varphi_2\subset\cdots$ is an increasing family of elements of $\N^\p$ with $\bigcup_{n\ge1}\varphi_n=\varphi$, then for every $x\in\varphi$ the combinatorial descendant of $x$ in $\varphi$ agrees with combinatorial descendant of $x$ in $\varphi_i$, \fa \suf large $i\ge1$.
\begin{proposition}
\label{contProp}
The considered lilypond line-segment model satisfies the continuity property. 
\end{proposition}
Section~\ref{contSec} is devoted to the proof of this proposition. Next, when considering the process of passing iteratively to combinatorial descendants, we need some control of distances between the corresponding geometric descendants. To be more precise, we consider a discretization of the Euclidean plane into large squares and call some of these squares good. Loosely speaking, if we start from any finite family of squares with the property that all adjacent squares are good, then these good squares should act as a shield: if starting from some point whose geometric descendant lies in the initial finite family of cubes, then the following geometric descendant is located either also in a square of that family or in an adjacent one. To be more precise, we say that the lilypond model satisfies the shielding condition (SH) if there exists a family of events $(A_s)_{s\ge1}$ on $\N_\M$ with $\lim_{s\to\infty}\P(X^{(1)}\cap Q_{3s}^\M(o) \in A_s)=1$ and \st the following condition is satisfied, where we write $B_1\oplus B_2=\{b_1+b_2:b_1\in B_1,\,b_2\in B_2\}$ for the Minkowski sum of $B_1,B_2\subset\R^2$.
\begin{enumerate}
\item[(SH)]
Consider the lattice $\Z^2$ with edges given by $\lcu \{ z_1,z_2\}:\lver z_1-z_2\rver_\infty\le1\rcu$, let $B\subset \Z^2$ and denote by $B^\p=\{z\in\Z^2\setminus B:|z-z^\p|_\infty= 1 \text{ for some }z^\p\in B\}$ the outer boundary of $B$. If $\varphi\in \N^\p$ is \st $(\varphi-sz^\p)\cap Q_{3s}^\M(o)\in A_s$ \fa $z^\p\in B^\p$, then 
$$\hge(\varphi,\hco(\varphi,x))\in sB\oplus Q_{3s}(o)$$
 \fa $x=(\xi,m)\in\varphi$ with $\hge(\varphi,x)\in sB\oplus Q_s(o)$.
\end{enumerate}
A site $z\in\Z^2$ with $(X^{(1)}-sz)\cap Q^\M_{3s}(o)\in A_s$, is called \emph{$s$-good}. See Figure~\ref{shFig} for an illustration of the shielding condition (SH). In Section~\ref{lilySec}, we verify that this condition is satisfied in the present setting.
\begin{proposition}
\label{shieldCond}
The considered lilypond line-segment model satisfies condition \emph{(SH)}.
\end{proposition}

\begin{figure}[!htpb]
     \centering
\begin{tikzpicture}[font=\footnotesize]
\fill[black!40!white] (2,2) rectangle (6,4);
\draw[step=2cm,gray,dashed] (0,0) grid (8,6);

\fill[black] (2.4,0.4) circle (1pt);
\draw[->] (2.4,0.4)--(2.4,2.6);
\draw[->] (1.8,2.6)--(6.2,2.6);
\fill[black] (1.8,2.6) circle (1pt);

\fill[black] (6.2,2.1) circle (1pt);
\fill[black] (6.1,3.2) circle (1pt);
\draw[->] (6.2,2.1)--(6.2,3.2);
\draw[->] (6.1,3.2)--(8.4,3.2);

\coordinate[label=180:${x}$] (u) at (2.4,0.4);
\coordinate[label=180:${\hge(\varphi\text{,}x)}$] (u2) at (3.2,2.8);
\coordinate[label=0:{${\hge(\varphi\text{,}\hco(\varphi\text{,}x))}$}] (u3) at (6.1,2.6);
\coordinate[label=180:${\hco(\varphi\text{,}x)}$] (u4) at (1.8,2.6);
\end{tikzpicture}
           \caption{Possible configuration as in condition (SH); set $sB\oplus Q_s(o)$ in gray}
           \label{shFig}
\end{figure}
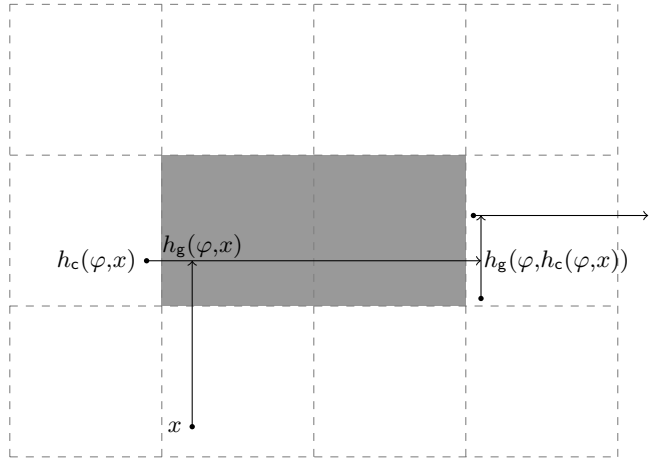

Finally, we need to know that in the lilypond model under consideration, sprinkled germs can be used to stop already existing segments from growing. To explain this property in greater detail, we first introduce the precise form of sprinkling that will be use in the following. For every $s>1$ the Poisson point process $X$ can be represented as $X=X^{(1)}(s)\cup X^{(2)}(s)$, where $X^{(1)}(s)$ is independent of $X^{(2)}(s)$ and both point processes are independently $\M$-marked homogeneous Poisson point processes with intensities $1-s^{-3}$ and $s^{-3}$, respectively. \Ip, $\lim_{s\to\infty}\P(X^{(2)}(s)\cap Q^\M_s(o)=\es)=1$.
Usually, the value of $s$ is clear from the context and then we write $X^{(i)}$ instead of $X^{(i)}(s)$. 

Having introduced the sprinkling, we now discuss a third important property of the lilypond model, which will be called \emph{uniform stopping property}. We assume that there exists a family of positive real numbers $(p_s)_{s\ge1}$ (possibly tending to $0$ as $s\to\infty$) with the following property. Ideally, we would like to achieve that, conditioned on $X^{(1)}\cap Q^\M_{3s}(o)$, with a probability at least $p_s$ adding the sprinkling $X^{(2)}\cap Q^\M_s(o)$ will cause all segments entering $Q_s(o)$ to become stuck in a cycle in $Q_s(o)$, whereas the structure of the lilypond model outside $Q_s(o)$ is left largely unchanged. However, this goal is too ambitious. Indeed, in some pathological cases, we can encounter realizations of $X^{(1)}\cap Q^\M_{3s}(o)$ for which the probability of observing a suitable sprinkling is much lower than $p_s$. Still, to prove the absence of forward percolation, it suffices to impose that the probability of such pathological configurations tends to $0$ as $s\to\infty$.
In order to state this property precisely, for $s>0$, $\varphi\in\N^\p$ and $z\in\Z^2$ it is convenient to denote by
$$\enviIn_{z,s}\(\varphi\)=\big\{ x\in\varphi\setminus Q^\M_s(o):\hge(\varphi,x)\in Q_s(sz)\big\}$$
the subset of all points $x\in\varphi\setminus Q^\M_s(o)$ whose geometric descendant is contained in $Q_s(sz)$. 

Now, we say that the lilypond satisfies the uniform stopping condition (condition (US)) if \te a family of positive real numbers $(p_s)_{s\ge1}$, $p_s\in(0,1]$ and a family of events $(A^\p_s)_{s\ge1}$ on $\N^\p\times \N^\p$ \st $A^\p_s\subset A_s\times\N^\p$, 
\begin{align}
\label{posChanceLem}
\P\(\big(X^{(1)}\cap Q^\M_{3s}(o),X^{(2)}\cap Q^\M_s(o)\big)\in A_s^\p\mid X^{(1)}\cap Q^\M_{3s}(o)\)\ge p_s 1_{X^{(1)}\cap Q^\M_{3s}(o)\in A_s}\text{ a.s.},
\end{align}
 and \st the following condition is satisfied.
\begin{enumerate}
\item[(US)]
Let $\varphi_1,\varphi_2\in \N^\p$ be \st $\varphi_2\subset Q^\M_s(o)$ and $(\varphi_1\cap Q_{3s}^\M(o),\varphi_2)\in A^\p_s$. Moreover, let $\psi\subset\R^{2,\M}\setminus Q^{\M}_{s}(o)$ be a finite set \st for every $z\in\Z^2$ either $\psi\cap Q^\M_{s}(sz)=\es$ or $\((\varphi_1-sz)\cap Q^\M_{3s}(o),(\psi-sz)\cap Q^\M_{s}(o)\)\in A^\p_s$. If $\varphi_1\cup\psi^\p\in\N^\p$ \fa $\psi^\p\subset\varphi_2\cup\psi$, then the following stabilization properties are true.
\begin{enumerate}
\item If $x\in \varphi_2$, then $\hco(\varphi_{1}\cup\varphi_2\cup\psi,x)=\hco(\varphi_2,x).$
\item If $x\in \varphi_1$, then either $\hco(\varphi_1\cup\varphi_2\cup\psi,x)\in\varphi_2$ or
\par\hspace*{-\leftmargin}\parbox{\textwidth}{$$\hco\(\varphi_1\cup\varphi_2\cup\psi,x\)=\hco\(\varphi_1\cup\psi,x\)\text{ and }x\not\in \enviIn_o\(\varphi_1\cup\psi\).$$}
\end{enumerate}
\end{enumerate}
If a site $z\in\Z^2$ is \st $\((X^{(1)}-sz)\cap Q^\M_{3s}(o),(X^{(2)}-sz)\cap Q^\M_{s}(o)\)\in A^\p_s$, then we also say that the site $z$ (or the sprinkling at this site) is \emph{$s$-perfect}. See Figure~\ref{uStopFig} for an illustration of the uniform stopping condition. Again, the verification of condition (US) is postponed to Section~\ref{lilySec}.
\begin{proposition}
\label{stopCond}
The considered lilypond line-segment model satisfies condition \emph{(US)}.
\end{proposition}

\begin{figure}[!htpb]
     \centering
     \begin{subfigure}[b]{0.45\textwidth}
           \centering
\begin{tikzpicture}
\draw (0,0) rectangle (5,5);
\draw[->] (3,6)--(3,1);
\draw[->] (2,1)--(4,1);
\draw[->] (4,0.5)--(4,2);
\draw[->] (3.5,2)--(5.5,2);
\end{tikzpicture}

           \caption{Configuration before addition of $\varphi_2$}
           \label{uStopFig1}
     \end{subfigure}%
     ~ 
     \begin{subfigure}[b]{0.45\textwidth}
           \centering
\begin{tikzpicture}
\draw (0,0) rectangle (5,5);

\draw[->] (3,6)--(3,3.55);
\draw[->] (2,1)--(4,1);
\draw[->] (4,0.5)--(4,2);
\draw[->] (3.5,2)--(5.5,2);

\draw[red,->] (2.7,3.55)--(3.35,3.55);
\draw[red,->] (3.35,3.7)--(3.35,3.0);
\draw[red,->] (3.6,3.0)--(2.85,3.0);
\draw[red,->] (2.85,2.9)--(2.85,3.55);

\end{tikzpicture}

           \caption{Configuration after addition of $\varphi_2$ (red)}
           \label{uStopFig2}
     \end{subfigure}
\caption{Possible configurations as in condition (US); $\psi=\es$}
\label{uStopFig}
\end{figure}
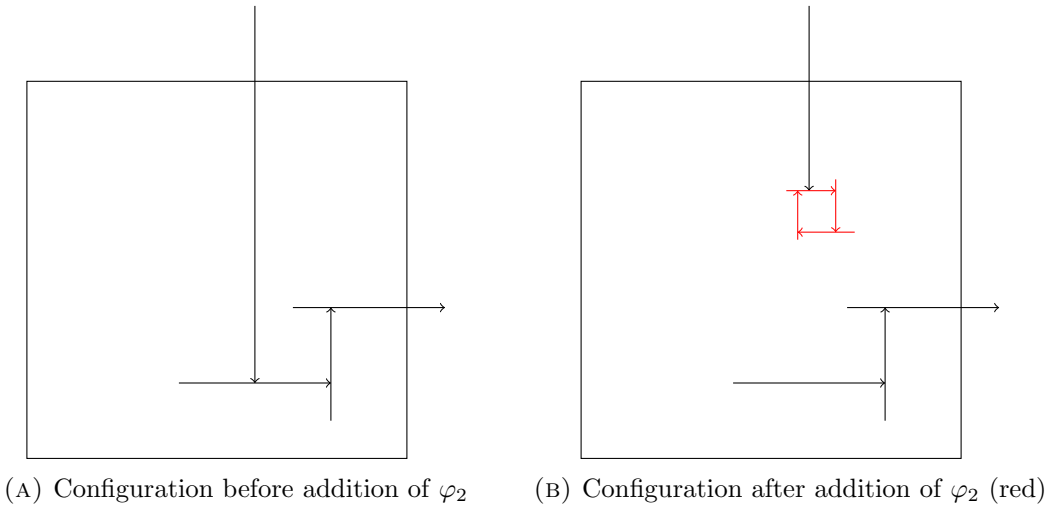

\begin{remark}
Properties (a) and (b) seem complicated at first sight, but allow for a simple heuristic description. 

Property (a) can be rephrased as stating that the lilypond model on $\varphi_2$ is autonomous in the sense that changes outside of $Q^\M_s(o)$ cannot alter this sub-configuration. To be more precise, as $\hco(\varphi_{1}\cup\varphi_2\cup\psi,x)=\hco(\varphi_2,x)$ and as $\varphi_2$ is assumed to be contained in $Q^\M_s(o)$, we conclude that $\hco(\varphi_2,x)\in Q^\M_s(o)$, \sot the iterates of $x\in\varphi_2$ stay in $Q^\M_s(o)$. It is also useful to note that from $x\in\varphi_2$ and $\hco(\varphi_2,x)\in Q^\M_s(o)$ we can deduce that $\hge(\varphi_2,x)\in Q_s(o)$.  

Property (b) yields the existence of suitable configurations \st any line segment that enters $Q_s(o)$ has a descendant in the sprinkled set (and therefore stops inside this square), whereas for any other line segment the addition of the sprinkled germs does not change the descendant. Also note that in the first case of part (b) knowing that $\hge(\varphi_{1}\cup\varphi_2\cup\psi,x)\in Q_s(o)$ for all $x\in\varphi_2$ allows us to deduce from $\hco\(\varphi_1\cup\varphi_2\cup\psi,x\)\in \varphi_2$ that $\hge\(\varphi_1\cup\varphi_2\cup\psi,x\)\in Q_s(o)$.
\end{remark}
\begin{remark}
Of course, the strong degree of internal stability that is required in properties (a) and (b)  occurs rather rarely, but condition (US) only requires that for most configurations induced by $X^{(1)}$ it occurs with a positive probability that is bounded away from $0$. 
\end{remark}

A rough sketch of the proof of Theorem~{\hyperref[part2]{\ref*{mainProp}.\ref*{part1}}} goes as follows. We start by considering the lilypond model on $X^{(1)}$. For large $s$, all but a sub-critical set of sites are $s$-good and therefore the configuration in the corresponding squares will only be influenced by sprinkling close to these squares. First, we add those sprinkled points whose effects we cannot control well in the sense that  $A^\p_s$ is not satisfied. This will increase slightly the sub-critical clusters formed by those squares for which we only have little information about the behavior of the lilypond model. However, since the sprinkling is of very low intensity, these enlarged clusters are still sub-critical. So far, we have held back the sprinkling inside the squares satisfying $A^\p_s$ and due to their special nature we can precisely control the effects of adding them to the system. \Ip, any purported infinite path must also be infinite before adding the final sprinkling. However, a path in the lattice that is killed with probability bounded away from $0$  each time it hits a site in the super-critical cluster, will be killed eventually. 
The preceding argument is made rigorous in Section~\ref{ideaSec}. Moreover, it can also be used to see that the number iterations until a cycle is reached exhibits at least exponentially decreasing tail probabilities.

\section{Absence of forward percolation}
\label{ideaSec}
In this section, we provide the details for the proof of Theorem~{\hyperref[part2]{\ref*{mainProp}.\ref*{part1}}}, which is based on a sprinkling argument. The main difficulty arises from the observation that the sprinkling has to be analyzed in a rather delicate way because two essential properties must be satisfied. On the one hand, we have to guarantee that if $x\in X^{(1)}$ is an element with $\# \hco^{(\infty)}({X^{(1)}},x)=\infty$, then after the sprinkling we must have $\# \hco^{(\infty)}({X},x)<\infty$. On the other hand, the sprinkling should not influence the lilypond model too strongly, since \fa $x\in X^{(1)}$ with $\# \hco^{(\infty)}({X^{(1)}},x)<\infty$ it has to be ensured that, after adding the sprinkling, the set of descendants $\hco^{(\infty)}({X},x)$ is still finite. As indicated above, we solve this problem by adding the sprinkling in two steps. First, we construct a point process $X^{(3)}\subset\R^{2,\M}$ with $X^{(1)}\subset X^{(3)}\subset X$ by adding only those germs of $X^{(2)}$ for which we have little knowledge as of how their addition would affect the existing directed graph. In a second step we add the remaining germs of $X^{(2)}$ for which we have precise information about their impact on the already existing model. 

To construct $X^{(3)}$, we introduce a discrete site process that allows us to determine whether there is either
\begin{enumerate}
\item no sprinkling inside the corresponding square, or
\item an $s$-perfect sprinkling, or
\item some other sprinkling.
\end{enumerate}
To be more precise, we define a $\{0,1,2\}$-valued site process $\lcu Y_z\rcu_{z\in\Z^2}$ as follows. If $X^{(2)}\cap Q^\M_s(sz)=\es$ and $z$ is $s$-good, then $Y_z=0$. Next, $Y_z=1$ if $z$ is $s$-perfect, i.e., if
$$\big((X^{(1)}-sz)\cap Q_{3s}^\M(o),(X^{(2)}-sz)\cap Q^\M_s(o)\big)\in A^{\p}_s.$$ 
If neither $Y_z=0$ nor $Y_z=1$, then $Y_z=2$. Since we assumed that $\N^\p$ does not contain the empty configuration, $z$ being $s$-perfect implies that $X^{(2)}\cap Q^\M_s(sz)\ne\es$, \sot there is no ambiguity in the definition of $Y$. Also note that conditioned on $X^{(1)}$ the site process $\{ Y_z\}_{z\in\Z^2}$ is an inhomogeneous independent site process.

In the next step we identify a large set of sites for which we have good control over the effect of the sprinkling. We recursively construct sets of revealed sites $(R^{(i)})_{i\ge0}$, $R^{(i)}\subset{\Z^2}$ and bad sites $(B^{(i)})_{i\ge0}$, $B^{(i)}\subset{\Z^2}$ as follows. Initially, put $R^{(0)}=B^{(0)}=\{ z\in\Z^2:Y_z=2\}$. Now \su $i\ge0$ and that $R^{(i)}\subset\Z^2$ as well as $B^{(i)}\subset\Z^2$ have already been constructed. For $z\in\Z^2$ write $S(z)=\{z^\p\in\Z^2: \left|z-z^\p\right|_\infty\le1\}$. Choose the closest bad site $z\in\Z^2$ to the origin with the property that its neighborhood $S(z)$ is not already completely revealed, i.e., $z\in B^{(i)}$ but $S(z)\not\subset R^{(i)}$. If several sites have this property, we choose the lexicographically smallest one. Put $R^{(i+1)}=R^{(i)}\cup S(z)$ and $B^{(i+1)}=B^{(i)}\cup\{z^\p\in S(z): X^{(2)}\cap Q^\M_s(sz^\p)\ne\es\}$. Finally, put $R=\bigcup_{i\ge0}R^{(i)}$ and see Figure~\ref{rConstr} for an illustration of the construction of the set $R$.
\begin{figure}[!htpb]
        \centering
        \begin{subfigure}[t]{0.28\textwidth}
                \centering
\begin{tikzpicture}[scale=0.85]
\draw[step=0.5cm,gray,dashed] (0,0) grid (5,5);
\fill (2.5,2.5) circle  (2pt);
\draw[pattern=crosshatch dots] (2,3) rectangle (2.5,3.5);
\draw[pattern=crosshatch dots] (4,0.5) rectangle (4.5,1);
\end{tikzpicture}

                \caption{Initial set of bad sites}
                \label{rFig1}
        \end{subfigure}%
\qquad
        \begin{subfigure}[t]{0.28\textwidth}
                \centering
\begin{tikzpicture}[scale=0.85]

\draw[step=0.5cm,gray,dashed] (0,0) grid (5,5);
\fill (2.5,2.5) circle  (2pt);
\fill[gray,opacity=0.5] (1.5,2.5) rectangle (3,4);
\draw[pattern=crosshatch dots] (2,3) rectangle (2.5,3.5);
\draw[pattern=crosshatch dots] (1.5,3.5) rectangle (2,4);
\draw[pattern=crosshatch dots] (4,0.5) rectangle (4.5,1);

\end{tikzpicture}

                \caption{New bad site revealed}
                \label{abcabc}
        \end{subfigure}
\qquad
        \begin{subfigure}[t]{0.28\textwidth}
                \centering
\begin{tikzpicture}[scale=0.85]

\draw[step=0.5cm,gray,dashed] (0,0) grid (5,5);
\fill (2.5,2.5) circle  (2pt);
\fill[gray,opacity=0.5] (1.5,2.5) rectangle (3,4);
\fill[gray,opacity=0.5] (1.0,4.0) rectangle (2.5,4.5);
\fill[gray,opacity=0.5] (1.0,4.0) rectangle (1.5,3.0);
\draw[pattern=crosshatch dots] (2,3) rectangle (2.5,3.5);
\draw[pattern=crosshatch dots] (1.5,3.5) rectangle (2,4);
\draw[pattern=crosshatch dots] (4,0.5) rectangle (4.5,1);
\end{tikzpicture}
                \caption{No new bad sites revealed}
                \label{abcabd}
        \end{subfigure}

                \caption{Construction of $R$}
\label{rConstr}
\end{figure}
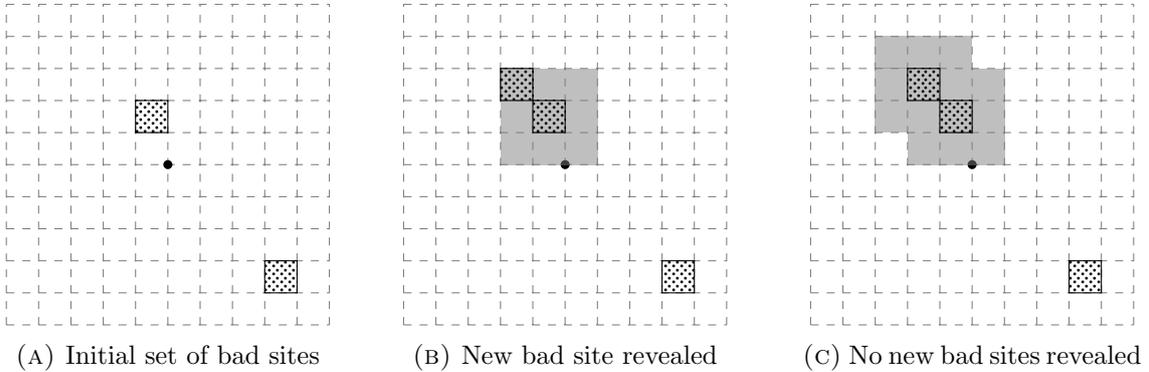

We first note that for \suf large $s\ge1$ only very few sites are revealed at all.
\begin{lemma}
\label{revLem}
There exists $s\ge1$ \st with probability $1$, the revealed sites $R\subset\Z^2$ are dominated from above by a sub-critical Bernoulli site-percolation process.
\end{lemma}
\begin{proof}
First, $R\subset Y^{(a)}\cup Y^{(b)}$, where $Y^{(a)}\subset\Z^2$ denotes the set of sites $z\in\Z^2$ whose neighborhood $S(z)$ contains a site which is not $s$-good and where $Y^{(b)}\subset\Z^2$ consists of those $z\in\Z^2$ with $X^{(2)}\cap Q^\M_{3s}(sz)\ne\es$. We note that $Y^{(a)}$ and $Y^{(b)}$ are $5$-dependent site processes that are independent of each other. Moreover, the probability that a given site is contained in $Y^{(a)}\cup Y^{(b)}$ tends to $0$ by the definition of $X^{(2)}$ and the assumption $\lim_{s\to\infty}\P\(X^{(1)}\cap Q_{3s}^\M(o) \in A_s\)=1$. Hence, the claim follows from~\cite[Theorem 0.0]{domProd}.
\end{proof}
In the remaining part of this section, we fix $s\ge1$ \st $R$ is dominated by a sub-critical Bernoulli site-percolation process. Then, we define $X^{(3)}=X^{(1)}\cup\(X^{(2)}\cap (sR\oplus Q_s(o))\)$. In other words, to create the original point process $X$ from $X^{(3)}$ we only have to add the sprinkling $X^{(2)}$ in the unrevealed region $sR^c\oplus Q_s(o)$, where $R^c= \Z^2\setminus R$ denotes the complement of $R$ in $\Z^2$. 

In order to compare $\hco^{(\infty)}\({X^{(3)}},x\)$ and $\hco^{(\infty)}(X,x)$, it is important to understand the effect of adding the sprinkled nodes $X^{(2)}\cap (sR^c\oplus Q_s(o))$ in a step-by-step manner. 
\begin{lemma}
\label{itLem}
Let $R^c=\{z_1,z_2,\ldots\}$ be an arbitrary enumeration of $ R^c$ and put $X^{(2,i)}=\bigcup_{j=1}^i (X^{(2)}\cap Q^\M_s(sz_j))$ as well as $X^{(3,i)}=X^{(3)}\cup X^{(2,i)}$. Then, for every $i\ge0$ the following properties are satisfied.
\begin{enumerate}
\item Let $x\in X^{(3)}$. Then, either $\hco(X^{(3,i)},x)\in X^{(2,i)}\cap Q^\M_s(sz)$ for some $z\in R^c$ or $\hco(X^{(3,j)},x)=\hco(X^{(3)},x)$ \fa $j\in\{0,\ldots,i\}$.
\item Let $z\in R^c$, $x\in \enviIn_z(X^{(3)})$ and $\hco(X^{(3,j)},x)=\hco(X^{(3)},x)$ \fa $j\in\{0,\ldots,i\}$. Then $X^{(2,i)}\cap Q^\M_s(sz)=\es$.
\end{enumerate}
\end{lemma}
\begin{proof}
We prove the desired properties by induction on $i$, the case $i=0$ being clear. If $X^{(2)}\cap Q^\M_s\(s{z_{i+1}}\)=\es$, then we deduce immediately that $X^{(3,i+1)}=X^{(3,i)}$ and can apply the induction hypothesis to conclude the proof. Therefore, we may assume $X^{(2)}\cap Q^\M_s(s{z_{i+1}})\ne\es$. 

After these preliminary observations, we begin with the proof of the first statement. Applying the definition of $s$-perfectness for $z=z_{i+1}$, $\varphi_1=X^{(3)}$ and $\psi =X^{(2,i)}$, we note that if $\hco(X^{(3,i+1)},x)=\hco(X^{(3,i)},x)$, then this statement also follows from the induction hypothesis. Observing that property 1. is also true in the remaining case, where $\hco(X^{(3,i+1)},x)\in X^{(2)}\cap Q^\M_s\(s{z_{i+1}}\)$ 
completes the proof of the first statement. 

Next, we verify the second statement. For $z\ne z_{i+1}$ the claim follows directly from the induction hypothesis, \sot we can concentrate on the case $z=z_{i+1}$ and $X^{(2)}\cap Q^\M_s\(s{z_{i+1}}\)\ne\es$. From  $\hco(X^{(3,i)},x)=\hco(X^{(3)},x)$ we conclude that $x\in \enviIn_z(X^{(3,i)})$, \sot  $s$-perfectness of $z_{i+1}$ implies $\hco(X^{(3,i+1)},x)\in X^{(2)}\cap Q^\M_{s}(s{z_{i+1}})$, contradicting the assumption $\hco(X^{(3,i+1)},x)=\hco(X^{(3)},x)$.
\end{proof}

The following result allows us to pass to the limit $i\to\infty$.

\begin{lemma}
\label{contLem}
 Let $x\in X^{(3)}$. Then, either $\hco(X,x)=\hco(X^{(3,i)},x)$ for all $i\ge0$ or $\hco(X,x)\in X^{(2)}\cap Q^\M_s(sz)$ for some $z\in R^c$.
Moreover, if $z\in R^c$, $x\in \enviIn_z(X^{(3)})$ and $\hco(X,x)=\hco(X^{(3)},x)$, then $X^{(2)}\cap Q^\M_s(sz)=\es$.
\end{lemma}

\begin{proof}
By continuity, $\hco(X,x)=\hco(X^{(3,i)},x)$ \fa sufficiently large $i\ge1$. \Ip, part $(i)$ of Lemma~\ref{itLem} implies that either $\hco(X,x)=\hco(X^{(3,i)},x)=\hco(X^{(3,j)},x)$ \fa $j\in\{0,\ldots,i\}$ or $\hco(X,x)=\hco(X^{(3,i)},x)\in X^{(2)}\cap Q^\M_s(sz)$ for some $z\in R^c$. Combining this result with part $(ii)$ of Lemma~\ref{itLem} yields the second part of the assertion.
\end{proof}

After this preliminary work, it is straightforward to establish the following comparison between the sets $\hco^{(\infty)}({X^{(3)}},x)$ and $\hco^{(\infty)}(X,x)$.
\begin{lemma}
\label{auxLem2}
If $x\in X$ is \st $\# \hco^{(\infty)}(X,x)=\infty$, then $\hco^{(\infty)}(X,x)\subset X^{(3)}$ and $\hco^{(n)}(X,x)=\hco^{(n)}(X^{(3)},x)$ \fa  $n\ge0$.
\end{lemma}
\begin{proof}
Let $x\in X$ be \st $\#\hco^{(\infty)}(X,x)=\infty$.
If \tes $n\ge0$ with $\hco^{(n)}(X,x)\in X^{(2)}\cap Q^\M_s(sz)$ for some $z\in R^c$, then $\#\hco^{(\infty)}(X,x)<\infty$. Indeed, choosing $i\ge0$ \sot $\hco(X^{(3,i)},\hco^{(n)}(X,x))=\hco^{(n+1)}(X,x)$, we can apply part (a) of property (US) to $\varphi^{(1)}=X^{(3)}$, $\psi=X^{(2,i)}\setminus Q^\M_s(sz)$ and $\varphi_2=X^{(2)}\cap Q^\M_s(sz)$ to deduce that $\hco^{(n+1)}(X,x)\in X^{(2)}\cap Q^\M_s(sz)$. Hence, using induction, we conclude that $\hco^{(m)}(X,x)\in X^{(2)}\cap Q^\M_s(sz)$ \fa $m\ge n$, which implies that $\#\hco^{(\infty)}(X,x)<\infty$. This observation yields $\hco^{(\infty)}(X,x)\subset X^{(3)}$ and the first statement in Lemma~\ref{contLem} allows us to conclude that $\hco^{(n)}(X,x)=\hco^{(n)}(X^{(3)},x)$ \fa $n\ge0$, as desired.
\end{proof}
\begin{proof}[Proof of Theorem~{\hyperref[part2]{\ref*{mainProp}.\ref*{part1}}}]
Assume the contrary. Then, by Lemma~\ref{auxLem2}, \tes $x\in X^{(3)}$ with
 $\# \hco^{(\infty)}(X^{(3)},x)=\infty$ and  $\hco^{(n)}(X,x)=\hco^{(n)}({X^{(3)}},x)$ \fa $n\ge0$.
Denote this event by $A^{*}_x$. It suffices to show that $\P\(A^{*}_x\mid X^{(1)},X^{(3)}, R\)=0$ \fa $x\in X^{(3)}$.

Putting $\hge^{(n)}(X^{(3)},x)=\hge\big(X^{(3)},\hco^{(n-1)}(X^{(3)},x)\big)$, we consider the sequence of geometric descendants $\big(\hge^{(n)}(X^{(3)},x)\big)_{n\ge1}$. First, we assert that $\big(\hge^{(n)}(X^{(3)},x)\big)_{n\ge1}$ hits infinitely many squares of the form $Q_s(sz)$ with $z\in R^c$. If $z,z^\p\in\Z^2$ are \st $\hge^{(n)}(X^{(3)},x)\in Q_s(sz)$ and $\hge^{(n+1)}(X^{(3)},x)\in Q_s(sz^\p)$, then applying condition (SH) with the set $B$ chosen as the connected component of $\{z\}\cup R$ containing $z$ shows that 
 $z^\p\in B\oplus Q_3(o)$ (noting that Lemma~\ref{revLem} implies the finiteness of $B$). \Ip, if $\hge^{(n+1)}(X^{(3)},x)$ does not lie in $sB\oplus Q_s(o)$, then $z^\p$ is contained in the outer boundary of $B$, which forms a subset of $R^c$. Hence, if $\#\hco^{(\infty)}\({X^{(3)}},x\)=\infty$, then after performing finitely many steps we obtain a geometric descendant contained in $sR^c\oplus Q_s(o)$. Since $\big(\hge^{(n)}(X^{(3)},x)\big)_{n\ge1}$ hits each bounded Borel set only a finite number of times, this proves the assertion. Hence, \te infinitely many $z_{1},z_{2},\ldots\in R^c$ and $i_1,i_2,\ldots\ge1$ \st $\hco^{(i_j)}(X^{(3)},x)\in\enviIn_{z_j}(X^{(3)})$ \fa $j\ge1$.  Moreover, we note that Lemma~\ref{contLem} implies $X^{(2)}\cap Q^\M_s(s{z_{j}})=\es$ \fa $j\ge1$.

However, we also \ob that conditioned on $X^{(1)}$, $X^{(3)}$ and $R$ the restriction of the site process $\{ Y_z\}_{z\in\Z^2}$ to $ R^c$ defines a $\{0,1\}$-valued Bernoulli site process \st for $z\in R^c$ the (conditional) probability of the event $\{ Y_z=1\}$ is given by $\P\(Y_z=1\mid Y_z\in\{0,1\},X^{(1)} \)$. \Ip, the events $\{ Y_{z_j}=1\}$, $j\ge1$ occur independently given $X^{(1)}$, $X^{(3)}$ and $R$  and by~\eqref{posChanceLem} \wh $\P\big(Y_z=1\mid Y_z\in\{0,1\},X^{(1)} \big)\ge p_s$ a.s. Therefore, with probability $1$, \tes $j_0\ge1$ with $Y_{z_{j_0}}=1$ contradicting the previously derived $X^{(2)}\cap Q^\M_s(sz_{j_0})=\es$. 
\end{proof}

\section{Proofs of auxiliary results}
\label{lilySec}
In the present section we provide the proof of the three auxiliary results that were used in the proof of the absence of forward percolation, Propositions~\ref{contProp},~\ref{shieldCond} and~\ref{stopCond}.

\subsection{Proof of Proposition~\ref{shieldCond}}
To begin with, we show that with high probability the length $\nu_1(Q_s(o)\cap I(X,x))$ of the intersection of a given square $Q_s(o)$ with any segment of the form $I(X,x)=[\xi,\hge(X,x)]$ is not too large.
\begin{lemma}
\label{extInfLem}
Let $\alpha>0$. Then, \tes a family of events $\big(A^{(1,\alpha)}_s\big)_{s\ge1}$ with 
$$\lim_{s\to\infty}\P\big(X^{(1)}\cap Q^\M_{s}(o)\in A^{(1,\alpha)}_s\big)=1$$ and \st the following property is satisfied. If $\varphi\subset Q^\M_s(o)$ is \st $\varphi\in A^{(1,\alpha)}_s$, then for every locally finite $\psi\subset \R^{2,\M}\setminus Q^\M_s(o)$ with $\varphi\cup\psi\in \N^\p$ and every $x\in \varphi\cup\psi$,
$$\nu_1\(I\(\varphi\cup \psi,x\)\cap Q_s(o)\)\le s^\alpha.$$ 
\end{lemma}

\begin{proof}
Without loss of generality, we may assume $\alpha<1$. We consider the cases $x\in \psi$ and $x\in \varphi$ separately and start with the case $x\in\psi$. By rotational and reflection symmetry, it suffices to prove that with high probability for every locally finite $\psi\subset\R^{2,\M}\setminus Q^\M_{s}(o)$ and $x=(\xi,v)\in\psi$ with $v=e_1$, $\pi_1(\xi)<0$ and $\pi_2\(\xi\)\in [0,s/2]$ \wh
$$\nu_1\big(I(X^{(1)}\cap Q^\M_{s}(o)\cup\psi,x)\cap Q_s(o)\big)\le s^\alpha.$$ 
Here $\pi_i:\R^{2}\to \R$ denotes the projection to the $i$th coordinate.
Put $\delta=s/\lfloor s^{1-\alpha/2}\rfloor$, $R_1=[-2.5\delta,2.5\delta]\times[-\delta,\delta]$ and $R_2=[-\delta/2,\delta/2]\times[-\delta,0]$. For $\xi\in\R^2$ we denote by $E^{}_{\xi}$ the intersection of the events $\varphi\cap ((\xi+R_2)\times \{e_2\})\ne\es$ and
$\varphi\cap ((\xi+R_1)\times \M)\subset(\xi+R_2)\times \{e_2\}$. See Figure~\ref{extInfFig} for an illustration.

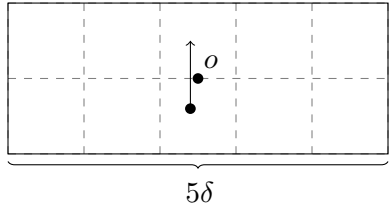
\begin{figure}[!htbp]
\centering
\begin{tikzpicture}[rotate=-90]
\draw[step=1cm,gray,dashed] (0,0) grid (2,5);
\draw (0,0) rectangle (2,5);
\fill[black] (1.4,2.4) circle (2pt);
\fill[black] (1.0,2.5) circle (2pt);
\draw[->] (1.4,2.4)--(0.5,2.4);
\coordinate[label=135:$o$] (u) at (1.0,2.9);
\draw[decorate,decoration=brace] (2.1,5.0)--(2.1,0.0);
\coordinate[label=right:$5\delta$] (u) at (2.5,2.2);
\end{tikzpicture}
\caption{Occurrence of $E_{o}$}
\label{extInfFig}
\end{figure}
Then, 
$$\P(X^{(1)}\cap Q^\M_{s}(o)\in E_\xi)\ge \exp\(-10(1-s^{-3})\delta^2\)\(1-\exp\(-(1-s^{-3})\delta^2/8\)\)\ge {\delta^2}/16$$ \fa $\xi\in\R^2$ and all $s>0$ sufficiently large. For $\sigma\in\R$ denote by $M^{}_{\sigma}\subset\R^2$ the set $\{(-s/2+5i\delta) e_1+\sigma e_2: 0\le i\le \lfloor s^\alpha/\(5\delta\)\rfloor\}$, \sot
$$
1-\P\big(X^{(1)}\cap Q^\M_{s}(o)\in \cup_{\xi\in M_\sigma^{}} E_{\xi}\big)\le \(1-\delta^2/16\)^{\lfloor s^\alpha/\(5\delta\)\rfloor+1}.
$$
Since $s^\alpha\delta^{}\ge s^{\alpha/2}$, we see that $\P\(X^{(1)}\cap Q^\M_{s}(o)\in \bigcap_{\xi\in M_{\sigma}^{}} E_{\xi}^c\)$ decays sub-exponentially fast as $s\to\infty$. Therefore also $\P\(X^{(1)}\cap Q^\M_{s}(o)\in \bigcup_{j=1}^{s/\delta-1}\bigcap_{\xi\in M_{-s/2+j\delta}^{}} E_{\xi}^c\)$ decays sub-exponentially fast in $s$. Note that $\varphi\in \bigcup_{\xi\in M_{-s/2+j\delta}} E_{\xi}$ implies that for every locally finite $\psi\subset\R^{2,\M}\setminus Q^\M_s(o)$ with $\varphi\cup\psi\in \N^\p$ and every $(\xi,v)\in\psi$ with $v=e_1$, $\pi_1(\xi)<0$ and $\pi_2(\xi)\in[-s/2+j\delta,-s/2+(j+1)\delta]$ \wh $\nu_1\(I\(\varphi\cup\psi,x\)\cap Q_s(o)\)\le s^\alpha$. This proves the first case of the claim.

Next, consider the case $x\in \varphi$. Using the Slivnyak-Mecke formula this part can be proven similarly as the case $x\in\psi$, but we include some details for the convenience of the reader. Again, by symmetry it suffices to prove that with high probability for every locally finite $\psi\subset\R^{2,\M}\setminus Q^\M_{s}(o)$ with $X^{(1)}\cap Q^\M_{s}(o)\cup\psi\in\N^\p$ and $x=(\xi,v)\in X^{(1)}\cap Q^\M_{s}(o)$ with $v=e_1$, $\pi_1(\xi)<s/2-s^\alpha$ and $\pi_2(\xi)\in (0,s/2)$ \wh
$$\nu_1\big(I(X^{(1)}\cap Q^\M_{s}(o)\cup\psi,x)\cap Q_s(o)\big)\le s^\alpha.$$ 
 For $\xi\in\R^2$ denote by $M^{\p}_{\xi}\subset\R^2$ the set $\{\xi+5i\delta e_1: 0\le i\le \lfloor s^\alpha/\(5\delta\)\rfloor \}$. 
Note that  $\varphi\in \bigcup_{\eta\in M_{\xi}^{\p}} E_{\eta}$ implies that for every locally finite $\psi\subset\R^{2,\M}\setminus Q^\M_s(o)$ \st $\varphi\cup\psi\in \N^\p$ \wh $\nu_1\(I\(\varphi\cup\psi,x\)\cap Q_s(o)\)\le s^\alpha$. 
 Moreover, by the Slivnyak-Mecke formula the expectation of the number $N$ of points $x=(\xi,e_1)\in X^{(1)}\cap Q^\M_{s}(o)$ for which the event $X^{(1)}\cap Q^\M_{s}(o)\in\bigcup_{\eta\in M_{\xi}^{}} E_{\eta}$ occurs is given by
$$\E N=\lambda \int_{Q^\M_{s}(o)}\P\Big(\big(X^{(1)}\cap Q^\M_{s}(o)\cup \{ (\xi,v)\}\big)\in\cup_{\eta\in M_{\xi}^{}} E_{\eta} \Big) \d (\xi,v).$$
By a similar argument as in the case $x\in \psi$, we see that the probability inside the integrand decays sub-exponentially fast in $s$ (uniformly over all $(\xi,v)\in Q^\M_{s}(o)$), which completes the proof Lemma~\ref{extInfLem}.
\end{proof}
\begin{remark}
A suitable analog of Lemma~\ref{extInfLem} can also be shown for isotropic line-segment models, but the proof becomes more involved. Indeed, a similar construction can be used, but now instead of four directions, one considers the shielding property seen from a set of directions with size growing polynomially in $s$.
\end{remark}
This auxiliary result immediately verifies condition (SH), when using the family of events $\big(A^{(1,\alpha)}_s\big)_{s\ge1}$ for some arbitrary $\alpha\in(0,1/2)$.
\begin{proof}[Proof of Proposition~\ref{shieldCond}]
Let $B\subset \Z^2$ be a finite set of sites and denote by $B^\p$ the outer boundary of $B$. Moreover, let $\varphi\in \N^\p$ be \st $(\varphi-sz^\p)\cap Q_{s}^\M(o)\in A^{(1,\alpha)}_s$ \fa $z^\p\in B^\p$ and $x=(\xi,v)\in\varphi$ be \st $\hge(\varphi, x)\in sB\oplus Q_s(o)$. Put $(\eta,w)=\hco(\varphi,x)$ and $D=\R^2\setminus (sB\oplus Q_{3s}(o))$. Using Lemma~\ref{extInfLem} twice implies that 
$\mathsf{dist}(\eta,D)\ge s/2$ and $\mathsf{dist}(\hge\(\varphi,(\eta,w)\),D)\ge s/4$, as desired.
\end{proof}

\subsection{Proof of Proposition~\ref{contProp}}
\label{contSec}
Next, we consider the property of continuity. This has already been investigated for a large class of lilypond models and is typically based on suitable descending chains arguments, see e.g.~\cite{lilypond6,lilypond5}. This approach also yields the desired result for the present anisotropic model, but for the convenience of the reader, we provide a detailed proof. 

In the following, for $\varphi\in\N^\p$ and $x=(\xi,v)\in\varphi$, it is convenient to write $f_\varphi(x)$ instead of $|\xi-\hge(\varphi,x)|$. First, we investigate how the behavior of $f_\varphi$ is related to the existence of long descending chains.
\begin{lemma}
\label{descChainDiffLem2}
Let $\varphi,\varphi^{\p}\in \N^\p$ and suppose that $x_1\in\varphi^{}\cap\varphi^{\p}$ is \st $f_{\varphi^{}}(x_1)< f_{\varphi^{\p}}(x_1)$. \Fm, define recursively $x_{i+1}\in\varphi^{}\cup\varphi^{\p}$ by 
$$x_{i+1}=\begin{cases}\hco\({\varphi^{}},x_i\)&\text{ if }x_i\in\varphi^{}\\ x_i&\text{else}\end{cases}$$
if $i$ is odd and by 
$$x_{i+1}=\begin{cases}\hco\({\varphi^{\p}},x_i\)&\text{ if }x_i\in\varphi^{\p}\\ x_i&\text{else}\end{cases}$$
if $i$ is even. Finally, put $i_0=\min\{i\ge1: x_i\not\in \varphi^{}\cap\varphi^{\p}\}$. Then, $(x_i)_{1\le i\le i_0}$ constitutes an $f_{\varphi^{}}(x_1)$-bounded anisotropic descending chain of pairwise distinct elements. Additionally, $f_{\varphi^{}}(x_i)<f_{\varphi^{\p}}(x_i)$ if $i\in\{1,\ldots i_0-1\}$ is odd and $f_{\varphi^{\p}}(x_i)<f_{\varphi^{}}(x_i)$ if $ i\in\{1,\ldots,i_0-1\}$ is even. \Ip, the absence of infinite anisotropic descending chains in $\varphi^{}$ and $\varphi^{\p}$ implies $i_0<\infty$.
\end{lemma}
\begin{proof}
At first glance, it might not be obvious how it is possible to have $f_{\varphi}(x_1)< f_{\varphi^\prime}(x_1)$ and $x_2\in\varphi\cap\varphi^\prime$. In other words, how can it be that the segment at $x_2$ stops the growth of the segment at $x_1$ in the lilypond model built from $\varphi$, but not in the one built from $\varphi^\p$. A more thorough thought reveals that this effect occurs if in the configuration $\varphi^\p$ the segment at $x_3$ stops the growth of the segment at $x_2$ before the latter can stop the growth of the segment at $x_1$.
Next, we extend this observation into a rigorous proof of the lemma and write $x_i=(\xi_i,v_i)$, $i\ge1$. The relation $\left|\xi_1-\xi_2\right|_\infty=f_{\varphi^{}}(x_1)$ follows immediately from the definition of stopping neighbors. Now, assume that $i\in\{2,\ldots,i_0-1\}$ is odd. From $x_i=\hco({\varphi^{\p}},x_{i-1})$ and $x_{i+1}=\hco(\varphi^{},x_{i})$ we conclude $|\xi_{i}-\xi_{i-1}|_\infty=f_{\varphi^{\p}}(x_{i-1})$ and $\left|\xi_{i+1}-\xi_{i}\right|_\infty=f_{\varphi^{}}(x_{i})$. \Fm, by induction $f_{\varphi^{\p}}(x_{i-1})<f_{\varphi^{}}(x_{i-1})$, \sot
\begin{align}
\label{descEq}
f_{\varphi^{}}(x_i)<\left|\xi_i-\hge(\varphi^{\p},x_{i-1})\right|<\min\(f_{\varphi^{\p}}(x_{i-1}), f_{\varphi^{\p}}(x_i)\).
\end{align}
The case of even $i\in\{2,\ldots,i_0-1\}$ is analogous.

It remains to show that the $\{x_i\}_{1\le i\le i_0}$ are pairwise disjoint and by the definition of $i_0$, it suffices to prove this claim for $\{x_i\}_{1\le i<i_0}$. \Fm, as $f_{\varphi^{}}(x_i)<f_{\varphi^{\p}}(x_i)$ if $1\le i<i_0$ is odd and $f_{\varphi^{\p}}(x_i)<f_{\varphi^{}}(x_i)$ if $1\le i<i_0$ is even, it suffices to consider the pairwise disjointness of the points in $\{x_i\}_{\substack{1\le i<i_0\\i\text{ even}}}$ and the points in $\{x_i\}_{\substack{1\le i<i_0\\i\text{ odd}}}$ separately. We consider for instance the case of even $i\in\{1,\ldots,i_0-1\}$. Then, we note that an application of~\eqref{descEq} and its analog for even parity yields
\begin{align*}
f_{\varphi^{\p}}(x_i)>f_{\varphi^{}}(x_{i+1})>f_{\varphi^{\p}}(x_{i+2}),
\end{align*}
\sot by induction $f_{\varphi^{\p}}(x_i)>f_{\varphi^{\p}}(x_j)$ \fa even $i,j\in\{1,\ldots,i_0-1\}$ with $i<j$.
\end{proof}
As corollary, we verify the continuity property of the lilypond model.
\begin{proof}[Proof of Proposition~\ref{contProp}]
Let $x\in\varphi$ be arbitrary and choose $n_0^\p\ge1$ \st $x\in\varphi_{n_0^\p}$. Lemma~\ref{descChainDiffLem2} implies that if $n\ge1$ and $s>0$ are \st $\varphi\cap Q^\M_s(o)=\varphi_n\cap Q^\M_s(o)$, but $\hco(\varphi,x)\ne\hco(\varphi_n,x)$, then \tes a $f_{\varphi}(x)$-bounded anisotropic descending chain starting in $x$ and leaving $Q_s(o)$. \Ip, if $n_1<n_2<\cdots$ is an increasing sequence with $\hco(\varphi,x)\ne\hco(\varphi_{n_i},x)$ \fa $i\ge1$, then there exist arbitrarily long $f_{\varphi}(x)$-bounded anisotropic chains starting at $x$. Since $\varphi$ is locally finite, this would produce an infinite $f_{\varphi}(x)$-bounded anisotropic chain starting at $x$, thereby contradicting the definition of $\N^\p$. 
\end{proof}
We conclude the investigation of the continuity property by showing that if $X\subset\R^{2,\M}$ is an independently and uniformly marked homogeneous Poisson point process, then $\P(X\in\N^\p)=1$. To obtain bounds for the probability of observing long anisotropic descending chains, the following auxiliary computation is useful.
\begin{lemma}
\label{intCompLem}
Let $b\ge0$, $k\ge0$ and $\xi_0\in\R^2$. Then, 
$\int_{\R^2} 1_{|\xi-\xi_0|_\infty\le b}|\xi-\xi_0|_\infty^{2k} \d x={4}b^{2k+2}/({k+1}).$
\end{lemma}
\begin{proof}
We may assume $\xi_0=o$ and put $\xi=(\xi^{(1)},\xi^{(2)})$. Then, by symmetry,
\begin{align*}
\int 1_{|\xi|_\infty\le b}|\xi|_\infty^{2k} \d \xi&=4\int_0^b\int_0^b\max\{\xi^{(1)},\xi^{(2)}\}^{2k}\d\xi^{(2)} \d\xi^{(1)}\\
&=8\int_0^b(\xi^{(1)})^{2k}\int_0^{\xi^{(1)}}\d\xi^{(2)}\d\xi^{(1)}\\
&=8\int_0^b(\xi^{(1)})^{2k+1}\d\xi^{(1)}\\
&={8}b^{2k+2}/({2k+2}).&\hspace{6.4cm}\qedhere
\end{align*}
\end{proof}

Using Lemma~\ref{intCompLem}, we can bound the probability of seeing long descending chains.

\begin{lemma}
\label{descChainCompLem}
Let $n\ge1$, $b,s>0$ and let $X$ be a homogeneous Poisson point process in $\R^2$ with intensity $\lambda>0$. Denote by $A^{(2)}_{s,b,n}$ the event that \te pairwise distinct $X_0,\ldots,X_n\in X$, with $X_0\in Q_s(o)$ and \st $\{X_i\}_{0\le i\le n}$ forms a $b$-bounded anisotropic descending chain. Then, $\P\big(A^{(2)}_{s,b,n}\big)\le{s^2(4b^2)^n\lambda^{n+1}}/n!$.
\end{lemma}

\begin{proof}
Denote by $N$ the number of $(n+1)$-tuples of pairwise distinct elements $X_0,\ldots,X_n\in X$  \st $X_0\in Q_s(o)$ and $\{X_i\}_{0\le i\le n}$ forms a $b$-bounded anisotropic descending chain. Then, using Lemma~\ref{intCompLem} and the Campbell formula,
\begin{align*}
\E N&\le\lambda^{n+1} \int \cdots\int 1_{\xi_0\in Q_s(o)}1_{b\ge \left|\xi_0-\xi_1\right|_\infty\ge\cdots\ge\left|\xi_{n-1}-\xi_n\right|_\infty}\d \xi_n\cdots \d \xi_0\\
&=4\lambda^{n+1}\int \cdots\int  1_{\xi_0\in Q_s(o)}1_{b\ge \left|\xi_0-\xi_1\right|_\infty\ge\cdots\ge\left|\xi_{n-2}-\xi_{n-1}\right|_\infty}\left|\xi_{n-2}-\xi_{n-1}\right|_\infty^{2}\d \xi_{n-1}\cdots \d \xi_0\\
&=\cdots\\
&=\frac{\(4b^2\)^n\lambda^{n+1}}{n!}\int  1_{\xi_0\in Q_s(o)}\d \xi_0=\frac{s^2\(4b^2\)^n\lambda^{n+1}}{n!}.&\hspace{1.6cm}\qedhere
\end{align*}
\end{proof}
Lemma~\ref{descChainCompLem} implies the absence of infinite anisotropic descending chains under Poisson assumptions.
\begin{corollary}
Let $X$ be an independently and uniformly $\M$-marked homogeneous Poisson point process in $\R^2$. Then, $\P(X\in\N^\p)=1$. 
\end{corollary}

\subsection{Proof of Proposition~\ref{stopCond}}
Finally, we verify the uniform stopping condition (US). In order to achieve this goal, it is first of all crucial to note that the configuration of the lilypond model in a given square is determined by the line segments entering this square. To make this more precise, it is convenient to introduce a variant of $\enviIn_{z,s} \varphi$ that also takes into account line segments intersecting the square $Q_{s}(sz)$. Hence, for $s>0$, $z\in\Z^2$ and $\varphi\in\N^\p$ we put
$$\enviInn_{z,s}\(\varphi\) =\{x\in \varphi\setminus Q_{s}^\M(sz): I(\varphi,x)\cap Q_{s}(sz)\ne\es\}.$$
\Fm, also line segments leaving a square will play an important role, \sot for $s>0$, $z\in\Z^2$ and $\varphi\in\N^\p$ we put
$$\enviOut_{z,s}\(\varphi\) =\{x\in \varphi\cap Q_{s}^\M(sz): \hge(\varphi,x)\not\in Q_s(sz)\}.$$
Since the value of $s$ is usually clear from the context, we often write $\enviInn_z\(\varphi\)$ and $\enviOut_z\(\varphi\)$ instead of $\enviInn_{z,s}\(\varphi\)$ and $\enviOut_{z,s}\(\varphi\)$. 
Using these definitions, we now obtain the following auxiliary result.
\begin{lemma}
\label{inAuxLem}
Let $s>0$ and $\varphi^{},\varphi^{\p}\in\N^\p$ be \st 
$\varphi^{}\cap Q^\M_s(o)\cup \enviInn_{o}\({\varphi^{}} \)=\varphi^{\p}\cap Q^\M_s(o)\cup \enviInn_{o}\({\varphi^{\p}}\).$
Then, 
\begin{enumerate}
\item $\hco\({\varphi^{}},x\)=\hco\({\varphi^{\p}},x\)$ \fa $x\in\varphi^{}\cap Q^\M_s(o)\cup\enviInn_{z}\({\varphi^{}}\)$ with $\{\hge({\varphi^{}},x),\hge({\varphi^{\p}},x)\}\cap Q_s(o)\ne\es$, 
\item $\enviInn_{o}\(\varphi^{}\)=\enviIn_o\(\varphi^{}\)$ if and only if $\enviInn_{o}\(\varphi^{\p}\)=\enviIn_o\(\varphi^{\p}\)$,
\item $\enviOut_{o}\(\varphi^{\p}\)=\enviOut_o\(\varphi^{}\)$.
\end{enumerate}
\end{lemma}
\begin{proof}
For readability, we write $f$, $f^\p$ instead of $f_{\varphi^{}}$, $f_{\varphi^{\p}}$. Put 
$$\varphi^\pp=\lcu x\in \varphi^{}\cap Q^\M_s(o)\cup\enviInn\varphi: \lcu \hge(\varphi^{},x), \hge(\varphi^{\p},x)\rcu \cap Q_s(o)\ne\es \rcu.$$ Our first goal is to show $f(x)=f^\p(x)$ \fa $x\in\varphi^\pp$. Suppose, for the sake of deriving a contradiction, that \tes $x_1\in\varphi^\pp$ with $f(x_1)\ne f^\p(x_1)$, e.g. $f(x_1)<f^\p(x_1)$. By Lemma~\ref{descChainDiffLem2} it suffices to show that for any such $x_1$ we have $x_2\in\varphi^\pp$, where $x_2=\hco(\varphi^{},x_1)$.

First, we assert that $\hge(\varphi^{},x_1)\in Q_s(o)$. Assuming the contrary, we could conclude from $x_1\in\varphi^\pp$ that $\hge(\varphi^{\p},x_1)\in Q_s(o)$. Since $f(x_1)<f^\p(x_1)$, we deduce that $\hge(\varphi,x_1)$ is contained on the line segment connecting $x_1$ and $\hge(\varphi^\p,x_1)$, which is only possible if $x_1\not\in Q^\M_s(o)$. \Ip, $x_1\in\enviInn_{o}(\varphi^{\p})$. However, as $\hge(\varphi^{},x_1)$ is not contained in $Q_s(o)$, we obtain that $x_1\not\in\enviInn_{o}(\varphi^{})$ contradicting our assumption $\enviInn_{o}(\varphi^{})=\enviInn_{o}(\varphi^{\p})$. This proves the assertion, which implies that $x_2\in \varphi^{}\cap Q^\M_s(o)\cup \enviInn_{o}(\varphi^{})$. From the assumption $f(x_1)<f^\p(x_1)$, we then conclude $f^\p(x_2)<f(x_2)$. We claim that $\hge(\varphi^{\p},x_2)\in Q_s(o)$ and assume the contrary for the sake of deriving a contradiction. Then, \wcon from $\hge(\varphi^\p,x_2)\in [x_2,\hge(\varphi^{},x_1)]$ and $\hge(\varphi,x_1)\in [x_2,\hge(\varphi,x_2)]$ that $x_2\not\in Q^\M_s(o)$, $x_2\not\in\enviInn_o(\varphi^\p)$ and $x_2\in\enviInn_o(\varphi)$, contradicting our assumption. This completes the proof of $f(x)=f^\p(x)$ \fa $x\in\varphi^\pp$. Property 2. is an immediate consequence of property 1.
To prove the third claim, let $x\in \varphi^{}\cap Q^\M_s(o)$. If $x\not\in\enviOut_o(\varphi)$, then $\hge(\varphi,x)\in Q_s(o)$ and therefore also $\hge(x,\varphi^\p)=\hge(x,\varphi)\in Q_s(o)$. In other words, $\varphi\cap Q^\M_s(o)\setminus \enviOut_o(\varphi)\subset\varphi^{\p}\cap Q^\M_s(o)\setminus \enviOut_o\(\varphi^\p\)$ and the other inclusion follows by symmetry. 
\end{proof}

Using Lemmas~\ref{extInfLem} and~\ref{descChainDiffLem2}, we show that the set $\enviInn_{o,s} (X)$ stabilizes with high probability.
\begin{lemma}
\label{extStab}
There exists a family of events $\big(A_s^{(3)}\big)_{s\ge1}$ \st 
$$\lim_{s\to\infty}\P\big(X^{(1)}\cap Q^\M_{3s}(o)\in A_s^{(3)}\big)=1$$
 and \st if $\varphi\in\N^\p$ is \st $\varphi\cap Q^\M_{3s}(o)\in A_s^{(3)}$, then $\enviInn_o \(\varphi\) =\enviInn_o\(\varphi\cup\psi\)$ \fa locally finite $\psi\subset\R^{2,\M}\setminus Q^\M_{3s}(o)$ with $\varphi\cup \psi \in \N^\p$.
\end{lemma}
\begin{proof}In the proof, we make use of the events $A^{(1,1/8)}_s$ and $A^{(2)}_{s,b,n}$ introduced in Lemmas~\ref{extInfLem} and~\ref{descChainCompLem}, respectively. \Fm, put $S=\{z\in\Z^2:|z|_\infty=2\}$ and denote by 
$$A^{(3^\p)}_s=A^{(1,1/8)}_s\cap \bigcap_{z\in S} \big\{\varphi\in \N_\M: \(\varphi-sz/3\)\cap Q^\M_{s/3}(o) \in A^{(1,1/8)}_{s/3}\big\} $$
the event that $A^{(1,1/8)}_s$ occurs in $Q_s(o)$ intersected with the event that $A^{(1,1/8)}_{s/3}$ occurs in each of the $(s/3)$-squares surrounding $Q_{s}(o)$. Now assume that $\varphi\in\N^\p$ is \st $\varphi\cap Q^\M_{3s}(o)\in A^{(3^\p)}_s$ and that \tes $\psi\subset\R^{2,\M}\setminus Q^\M_{3s}(o)$ with $\varphi\cup \psi\in \N^\p$ and $\enviInn_o \(\varphi\) \ne\enviInn_o\(\varphi\cup\psi\).$ Then, \tes $x\in Q^\M_{2s}(o)$ with $\hco\(\varphi,x\)\ne \hco\(\varphi\cup\psi,x\)$. By Lemma~\ref{descChainDiffLem2}, there exists an $s^{1/4}$-bounded anisotropic descending chain of pairwise distinct elements of $\varphi$ starting in $x$ and ending in $\R^2\setminus Q_{3s}(o)$. \Fm, by our assumption this chain consists of at least $n_s=\lfloor s/(2s^{1/4})\rfloor= \lfloor s^{3/4}/2\rfloor $ hops. Hence, $\varphi\cap Q^\M_{3s}(o) \in A^{(2)}_{2s, s^{1/4},n_s}$  and therefore, we put $A^{(3)}_s=A^{(3^\p)}_s\setminus A^{(2)}_{2s, s^{1/4},n_s}$. By Lemma~\ref{descChainCompLem}, the probability for the occurrence of $X^{(1)}\cap Q^{\M}_{3s}(o)\in A^{(2)}_{2s,s^{1/4},n_s}$ is bounded from above by $4s^2(4s^{1/2})^{n_s}/ {n_s}!$. By Stirling's formula, the latter expression tends to $0$ as $s\to\infty$.
\end{proof}

First, we provide a definition of $A_s$ \st if $\varphi\subset Q^\M_{3s}(o)$ is \st $\varphi\in A_s$, then 
\begin{enumerate}
\item $\varphi$ satisfies the shielding property, i.e., $\varphi\cap Q^\M_{s}(o)\in A^{(1,1/2)}_s,$
\item $\varphi$ satisfies the external stabilization property, i.e., $\varphi\in A^{(3)}_s,$ and
\item the points of $\varphi$ do not come too close to each other and also not too close to the boundary of $Q_s(o)$. 
\end{enumerate}
To be more precise, for $s\ge1$ we put 
$$A_s=A^{(1,1/2)}_s\cap A^{(3)}_s\cap A^{(4)}_s,$$ 
where $A^{(4)}_s$ denotes the family of all $\varphi\in \N_\M$ \st $\varphi\subset Q^\M_{3s}(o)$ and $\varphi$ is \emph{$s^{-4}$-separated}, i.e., $|\pi_k(\xi)-\pi_k(\eta)|\ge s^{-4}$ \fa $k\in\{1,2\}$ and all
$$x=(\xi,v),y=(\eta,w)\in (\varphi\cap Q^\M_{3s}(o))\cup\big(\{\pm(s/2,s/2)\}\times \M\big)\text{ with }x\ne y,$$
where we recall that $\pi_k(\xi)$, $k\in\{1,2\}$ denotes the $k$th coordinate of $\xi$. Note that the set $\{\pm(s/2,s/2)\}\times \M$ is added, since it is important to have some room at the boundary of $Q_s(o)$, where we can add sprinklling used to stop incoming segments.
Taking into account Lemmas~\ref{extInfLem} and~\ref{extStab}, in order to show that $\lim_{s\to\infty}\P(X^{(1)}\cap Q^\M_{3s}(o)\in A_s)=1$ it suffices to prove $\lim_{s\to\infty}\P\big(X^{(1)}\cap Q^\M_{3s}(o)\in A^{(4)}_s\big)=1$, which is achieved in the following result.
\begin{lemma}
As $s\to\infty$ the probability $\P\big(X^{(1)}\cap Q^\M_{3s}(o)\in A^{(4)}_s\big)$ tends to $1$. 
\end{lemma}
\begin{proof}
The expected number of distinct points $x=(\xi,v)$, $y=(\eta,w)\in X^{(1)}\cap Q^\M_{3s}(o)$ satisfying $|\pi_k(\xi)-\pi_k(\eta)|\le s^{-4}$ for some $k\in\{1,2\}$ is of order at most $s^3\cdot s^{-4}$ and therefore tends to $0$ as $s\to\infty$. Similarly, the expected number of $x=(\xi,v)\in X^{(1)}\cap Q^\M_{3s}(o)$ with $|\pi_k(\xi)-\zeta|\le s^{-4}$ for some $k\in\{1,2\}$ and $\zeta=\pm s/2$ is of order $s\cdot s^{-4}$, \sot it also tends to $0$ as $s\to\infty$. 
\end{proof}

The next step is to introduce the family of events $\(A^\p_s\)_{s\ge1}$, i.e., to define suitable sprinkling configurations. Here, small four-cycles play an important role.
\begin{definition}
\label{4cycle}
Let $\delta>0$ and $\xi\in\R^2$. We say that $D=\{ x_{1},\ldots,x_{4}\}=\{(\xi_{1},v_1),\ldots,(\xi_4,v_{4})\}\subset Q^{\M}_\delta(\xi)$ forms a \emph{$(\xi,\delta)$-cycle} if the following conditions are satisfied, where we put formally $x_5=x_1$ and $v_5=v_1$.
\begin{enumerate}
\item $v_j=\rho_{\pi/2}\(v_{j+1}\)$ \fa $j\in\{1,\ldots,4\}$,
 where $\rho_{\pi/2}:\R^2\to\R^2$ denotes rotation by $\pi/2$,
\item $\hge\(D,x_j\)\in Q_{\delta}(\xi)$ \fa $j\in\{1,\ldots,4\}$,
\item $\hge\(D,x_j\)\in I\(D,x_{j+1}\)$ \fa $j\in\{1,\ldots,4\}$, 
\item $\min\(\langle \xi-\xi_{j}, v_j\rangle,\langle \xi-\xi_{j}, v_{j+1}\rangle\)>0$ \fa $j\in\{1,\ldots,4\}$.
\end{enumerate}
The fourth condition ensures that $\xi$ belongs to the inner part delimited by the $(\xi,\delta)$-cycle. \Ip, if we consider a line segment starting from $\R^2\setminus Q_\delta(\xi)$ and whose corresponding ray contains $\xi$, then this ray must hit the cycle. See Figure~\ref{4cycleFig} for an illustration of a $(\xi,\delta)$-cycle.
\end{definition}
\begin{figure}[!htbp]
\centering
\begin{tikzpicture}
\draw[step=1cm,gray,dashed] (0,0) grid (5,5);
\draw (0,0) rectangle (5,5);
\fill[black] (3.1,4.3) circle (2pt);
\fill[black] (4.7,1.7) circle (2pt);
\fill[black] (1.9,0.7) circle (2pt);
\fill[black] (0.3,3.2) circle (2pt);
\fill[black] (2.5,2.5) circle (2pt);
\draw[->] (3.1,4.3)--(3.1,1.7);
\draw[->] (4.7,1.7)--(1.9,1.7);
\draw[->] (1.9,0.7)--(1.9,3.2);
\draw[->] (0.3,3.2)--(3.1,3.2);
\coordinate[label=right:$\xi$] (u) at (2.5,2.5);
\draw[decorate,decoration=brace] (5.1,5.0)--(5.1,0.0);
\coordinate[label=right:$\delta$] (u) at (5.1,2.5);
\end{tikzpicture}
\caption{Example of a $(\xi,\delta)$-cycle}
\label{4cycleFig}
\end{figure}
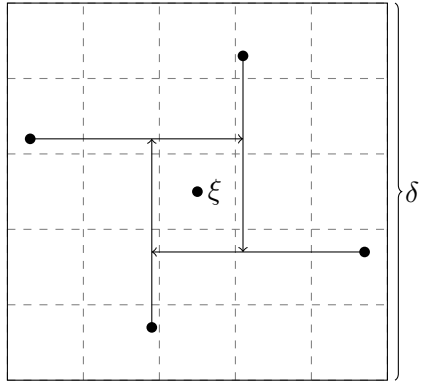
An important feature of $(\xi,\delta)$-cycles is the following strong external stabilization property.
\begin{lemma}
\label{stabCycle}
Let $\delta>0$, $\xi\in\R^2$ and $\varphi\in \N^\p$ be \st $\varphi\cap Q^\M_{3\delta}(\xi) =\es$. \Fm, let $D=\{x_1,\ldots,x_4\}=\{(\xi_1,v_1),\ldots,(\xi_4,v_4)\}\subset\R^{2,\M}$ be a $(\xi,\delta)$-cycle. Then, $\hco\({D\cup\varphi},x_i\)=\hco\(D,x_i\)$ \fa $i\in\{1,\ldots,4\}$. 
\end{lemma}
\begin{proof}
Suppose \tes $i\in\{1,\ldots,4\}$ with $\hco\(D\cup\varphi,x_i\)\ne \hco\(D,x_i\)$. By the hard-core property we see that we cannot have $f_{D\cup\varphi}(x_j)\ge f_D(x_j)$ \fa $j\in\{1,\ldots,4\}$ and we choose $j_1\in\{1,\ldots,4\}$ with $f_{D\cup\varphi}(x_{j_1})<f_{D}(x_{j_1})$. Since all segments grow at the same speed (which is equal to 1), we deduce that
\begin{align*}
|\eta-\xi|_\infty&\le |\eta-\xi_{j_1}|_\infty+|\xi_{j_1}-\xi|_\infty\le f_{D\cup\varphi}(x_{j_1}) +\delta/2\le 2\delta,
\end{align*}
where $(\eta,w)=\hco\({D\cup\varphi},x_{j_1}\)$. Thus, $\hco\({D\cup\varphi},x_{j_1}\)\in Q^\M_{3\delta}(\xi)$, violating the assumption $\varphi\cap Q^\M_{3\delta}(\xi)=\es$.
\end{proof}
The notion of $(\xi,\delta)$-cycles can be used to define configurations $X^{(2)}\cap Q^\M_s(o)$ satisfying the relation $\# \hco^{(\infty)}\({X^{(1)}\cup X^{(2)}\cap Q^\M_s(o)},x\)<\infty$ \fa $x\in\enviInn_o\({X^{(1)}}\)$. These cycles are used to stop all segments intersecting $Q_s(o)$ except for those leaving the square.
More precisely, we make the following definition, where for $\psi\in\N^\p$ we write $\psi\in A^{*}_{\xi,\delta}$ if the configuration of $\psi\cap Q_\delta(\xi)$ consists precisely of one $(\xi,\delta)$-cycle. 
\begin{definition}
\label{deviceDef}
Let $s>0$, $\varphi\in\N_\M$ and $\psi\in\N^\p$ be \st $\varphi\subset Q^\M_{3s}(o)$, $\varphi\in A^{(3)}_s$ and $\psi\subset Q^\M_s(o)$. Let $\varphi^\p\in\N^\p$ be any locally finite set with $\varphi^\p\cap Q^\M_{3s}(o)=\varphi$. Then, we put $\delta=s^{-4}$, 
\begin{align*}
M_1&=\{\xi+{3\delta}v/{8}:(\xi,v)\in \varphi\cap Q^\M_s(o)\setminus\partial^{\mathsf{out}}_o(\varphi^\p)\},\\
M_2&=\{P_{(\xi,v)}+{3\delta}v/{8}:(\xi,v)\in \enviIn_o(\varphi^\p)\},\text{ and }\\
M_3&=\{(-s/2+3\delta/8)(e_1+e_2)\}.
\end{align*}
Here $P_{(\xi,v)}$ denotes the first intersection point of $I(\varphi,(\xi,v))$ and $\partial Q_s(o)$. Note that since $\varphi\in A^{(3)}_s$, the definition of $M_1$, $M_2$ and $M_3$ does not depend on the choice of $\varphi^\p$. 
Then, we define $(\varphi,\psi)\in A^\p_s$ to be the intersection of the events $\lcu \psi\subset\((M_1\cup M_2\cup M_3)\oplus Q_{\delta/16}(o)\)\times\M\rcu$, $\varphi\in A_s$ and
$\psi\in \bigcap_{(\xi,v)\in M_1\cup M_2\cup M_3} A^{*}_{\xi,\delta/16}$.
See Figure~\ref{sprinkleFig} for an illustration of the effect on the lilypond model when adding a set of germs $\psi$ satisfying $(\varphi,\psi)\in A^{\p}_{s}$.
\end{definition}
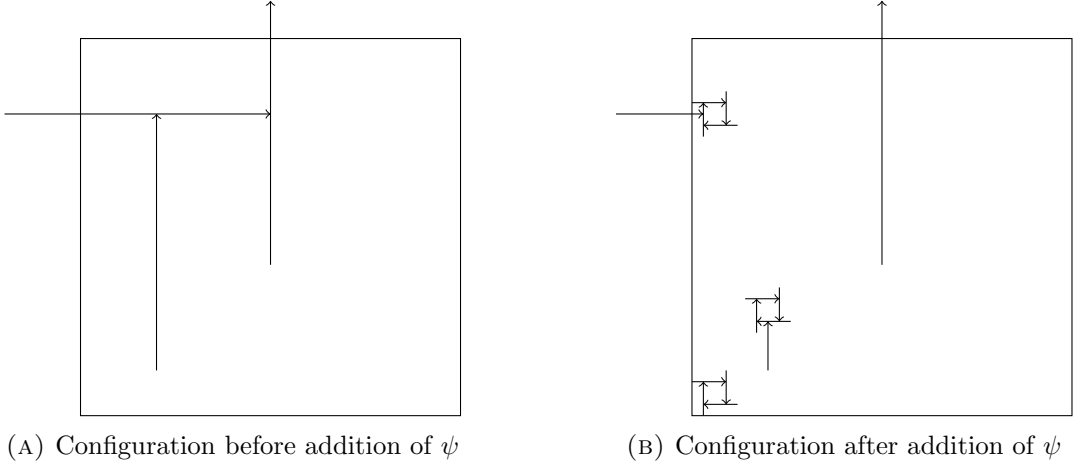
\begin{figure}[!htpb]
     \centering
     \begin{subfigure}[b]{0.5\textwidth}
           \centering
\begin{tikzpicture}
\draw (0,0) rectangle (5,5);
\draw[->] (-1,4)--(2.5,4);
\draw[->] (2.5,2)--(2.5,5.5);
\draw[->] (1,0.6)--(1,4);
\end{tikzpicture}

           \caption{Configuration before addition of $\psi$}
           \label{sprinkleFig1}
     \end{subfigure}%
     ~ 
     \begin{subfigure}[b]{0.5\textwidth}
           \centering
\begin{tikzpicture}
\draw (0,0) rectangle (5,5);
\draw[->] (-1,4)--(0.15,4);
\draw[->] (2.5,2)--(2.5,5.5);
\draw[->] (1,0.6)--(1,1.25);

\draw[->] (0.0,4.15)--(0.45,4.15);
\draw[->] (0.45,4.3)--(0.45,3.85);
\draw[->] (0.6,3.85)--(0.15,3.85);
\draw[->] (0.15,3.7)--(0.15,4.15);

\draw[->] (0.7,1.55)--(1.15,1.55);
\draw[->] (1.15,1.7)--(1.15,1.25);
\draw[->] (1.3,1.25)--(0.85,1.25);
\draw[->] (0.85,1.1)--(0.85,1.55);

\draw[->] (0.0,0.45)--(0.45,0.45);
\draw[->] (0.45,0.6)--(0.45,0.15);
\draw[->] (0.6,0.15)--(0.15,0.15);
\draw[->] (0.15,0.0)--(0.15,0.45);

\end{tikzpicture}

           \caption{Configuration after addition of $\psi$}
           \label{sprinkleFig2}
     \end{subfigure}
\caption{Impact of the addition of $\psi$ with $(\varphi,\psi)\in A^\p_s$}
\label{sprinkleFig}
\end{figure}

\begin{lemma}
\label{posChanceLemLil}
The events $\(A_s^\p\)_{s\ge1}$ introduced in Definition~\emph{\ref{deviceDef}} satisfy condition~\eqref{posChanceLem}.
\end{lemma}
\begin{proof}
Assume that $X^{(1)}\cap Q^\M_{3s}(o)\in A_s$ and let $M_1,M_2,M_3\subset Q_s(o)$ be as in the definition of the event $A^\p_s$. \Fm, put $\delta=s^{-4}$ and $M_{1,2,3}=M_1\cup M_2\cup M_3$. We conclude from $\delta$-separatedness that \fa $x_1,x_2\in M_{1,2,3}$ with $x_1\ne x_2$ \wh $Q_{\delta/16}(x_1)\cap Q_{\delta/16}(x_2)=\es$. \Ip, for every $(\xi,v)\in M_{1,2,3}$ the event $X^{(2)}\in A^{*}_{\xi,\delta/16}$ is independent of the family of events $X^{(2)}\in A^{*}_{\eta,\delta/16}$ for $(\eta,w)\in M_{1,2,3}$ with $\eta\ne \xi$. \Fm, from $\delta$-separatedness we also conclude $\#M_1+\#M_2 \le 3\lceil s\delta^{-1}\rceil$. Finally, note that $\P\(X^{(2)}\cap Q^\M_s(o)\subset (M_{1,2,3}\oplus Q_{\delta/16}(o))\times\M\)\ge \P\(X^{(2)}\cap Q^\M_s(o)=\es\)$. Hence, we may choose $p_s=\P\(X^{(2)}\cap Q^\M_s(o)=\es\)\P\big(X^{(2)}\in A^{*}_{o,\delta/16}\big(X^{(2)}_o\big)\big)^{3\lceil s\delta^{-1}\rceil+1}.$
\end{proof}

Finally, we verify condition (US). Note that if $(\varphi_1,\varphi_2)\in A^\p_s$, then $\varphi_2\ne\es$ is an immediate consequence of the definition of $A^\p_s$. Moreover, part (a) of the condition follows from Lemma~\ref{stabCycle}. Hence, it remains to verify part (b). This will be achieved in the following two results.
As a first step, we provide a precise description of the combinatorial descendant function $\hco(\varphi_1\cup\varphi_2\cup\psi,\cdot)$ under the additional assumption that $\enviIn_z(\varphi_{1}\cup\psi)=\enviIn_z (\varphi_{1})$ and $\enviOut_z(\varphi_{1}\cup\psi)=\enviOut_z(\varphi_{1})$ for all $z\in\Z^d$.
\begin{lemma}
\label{usLem1}
Let $z_0\in\Z^2$, $\varphi_1,\varphi_2\in \N^\p$ be \st $\varphi_2\subset Q^\M_s(sz_0)$ and  $((\varphi_1-sz_0)\cap Q_{3s}^\M(o),\varphi_2-sz_0)\in A^\p_s$. Moreover, let $\psi\subset\R^{2,\M}\setminus Q^{\M}_{s}(sz_0)$ be a finite set \st for every $z\in\Z^2$ either $((\varphi_1-sz)\cap Q^\M_{3s}(o),(\psi-sz)\cap Q^\M_{s}(o))\in A^\p_s$ or $\psi\cap Q^\M_{s}(sz)=\es$. \Fm, assume that $\varphi_1\cup\psi^\p\in\N^\p$ \fa $\psi^\p\subset\varphi_2\cup\psi$ and also that $\enviIn_z (\varphi_{1}\cup \psi)=\enviIn_z (\varphi_{1})$ and $\enviOut_z( \varphi_{1}\cup\psi)=\enviOut_z (\varphi_{1})$ \fa $z\in\Z^2$.
Then, for every $x\in \varphi_1\cup\varphi_2\cup\psi$,
\begin{align*}
\hco(\varphi_1\cup\varphi_2\cup\psi,x)=
\begin{cases}
\hco(\varphi_2,x) &\text{if }x\in\varphi_2,\\
\hco(\varphi_1\cup\varphi_2,x)&\text{if }\hge(\varphi_{1}\cup\psi,x)\in Q_s(sz_0),\\
\hco(\varphi_{1}\cup\psi,x) &\text{otherwise.}
\end{cases}
\end{align*}
\end{lemma}
\begin{proof}
Without loss of generality, we may assume that $z_0=o$. Define a function $\hco^\p: \R^{2,\M}\to \R^{2,\M}$ by the right hand side of the asserted identity in the statement of the lemma. We show that $\hco^\p$ satisfies the characteristic properties of the lilypond model on $\varphi_1\cup\varphi_2\cup\psi$ and therefore coincides with $\hco(\varphi_{1}\cup\varphi_2\cup\psi,\cdot)$. The geometric descendant function corresponding to $\hco^\p$ is denoted by $\hge^\p$. 
First, we note that by definition of $A^\p_s$, the hard-core property and the existence of stopping neighbors is clearly satisfied for every $x\in \varphi_2$. Next, we claim that for every $x\in\varphi_{1}\cup\psi$,
\begin{align}
\label{shorterHC}
|\xi - \hge^\p(x)|\le |\xi -\hge(\varphi_{1}\cup\psi,x)|. 
\end{align}
This will imply the hard-core property. To prove~\eqref{shorterHC} it suffices to consider the case where $\hge(\varphi_{1}\cup\psi,x)\in Q_s(o)$.  Assume the contrary, i.e., that $|\xi - \hge^\p(x)|>|\xi -\hge\(\varphi_{1}\cup\psi,x\) |$. Then, by properties $\enviIn_o(\varphi_{1}\cup\psi)=\enviIn_o(\varphi_{1})$, $\enviOut_o(\varphi_{1}\cup\psi)=\enviOut_o(\varphi_{1})$ and the definition of $A_s^\p$, this would imply that $(\eta,w)=\hco(\varphi_1\cup\psi,x)\in \R^{2,\M}\setminus Q^\M_{3s}(o)$, 
contradicting 
$|\eta-\hge(\varphi_1\cup\psi,x)|<|\xi-\hge(\varphi_1\cup\psi,x)|.$

Next, we consider the issue of existence of stopping neighbors and put $y=\hco\(\varphi_{1}\cup\psi,x\)$. If $\hge\(\varphi_{1}\cup\psi,x\)\in Q_s(o)$, then this follows again from the properties $\enviIn_o \(\varphi_{1}\cup\psi\)=\enviIn_o \(\varphi_{1}\)$ and $\enviOut_o \(\varphi_{1}\cup\psi\)=\enviOut_o\( \varphi_{1}\)$. If $x\in \enviOut_o\(\varphi_{1}\cup\psi\)$, then an elementary geometric argument shows that $\hge(\varphi_{1}\cup\psi,y)\not\in Q_s(o)$, \sot $\hge(\varphi_{1}\cup\psi,x)\in I(\varphi_{1}\cup\psi,y)$. It remains to consider the case, where $x\in \R^{2,\M}\setminus Q^\M_s(o)$, but $x\not\in \enviIn_o\(\varphi_{1}\cup\psi\)$. Then, $y$ is clearly a stopping neighbor of $x$ with respect to $\hco^\p$ if $\hge(\varphi_{1}\cup\psi,y)\not\in Q_s(o)$.
\Fm, the case $y\in\(\varphi_{1}\cup\psi\)\cap Q^\M_s(o)\setminus \enviOut_o \(\varphi_{1}\cup\psi\)$ is not possible, as it would imply $\hge(\varphi_1\cup\xi,x)\in Q_s(o)$. Finally, consider the case $y\in\enviIn_o \(\varphi_{1}\cup\psi\)$ and denote by $P_y$ the first intersection point of $I(\varphi_1\cup\psi,y)$ and $Q_s(o)$. Then, the claim follows from the observation $\hge(\varphi_{1}\cup\xi,x)\in [\eta,P_y)$.
\end{proof}

Using Lemma~\ref{usLem1}, we can now complete the verification of condition (US).

\begin{lemma}
\label{usLem2}
Let $z_0\in\Z^2$, $\varphi_1,\varphi_2\in \N^\p$ be \st $\varphi_2\subset Q^\M_s(sz_0)$ and  $((\varphi_1-sz_0)\cap Q_{3s}^\M(o),\varphi_2-sz_0)\in A^\p_s$. Moreover, let $\psi\subset\R^{2,\M}\setminus Q^{\M}_{s}(sz_0)$ be a finite set \st for every $z\in\Z^2$ either $((\varphi_1-sz)\cap Q^\M_{3s}(o),(\psi-sz)\cap Q^\M_{s}(o))\in A^\p_s$ or $\psi\cap Q^\M_{s}(sz)=\es$. \Fm, assume that $\varphi_1\cup\psi^\p\in\N^\p$ \fa $\psi^\p\subset\varphi_2\cup\psi$. Then, for every $z\in\Z^2$, 
\begin{enumerate}
\item $\enviIn_z(\varphi_1\cup\varphi_2\cup\psi)=\enviIn_z(\varphi_{1})$.
\item $\enviOut_z(\varphi_1\cup\varphi_2\cup\psi)=\enviOut_z(\varphi_{1})$.
\item Let $x\in \varphi_1$. Then, either $\hco(\varphi_{1}\cup\varphi_2\cup\psi,x)\in\varphi_2$
or
\par\hspace*{-\leftmargin}\parbox{\textwidth}{$$\hco(\varphi_{1}\cup\varphi_2\cup\psi,x)=\hco(\varphi_1\cup \psi,x)\text{ and }x\not\in \enviIn_{z_0}(\varphi_1\cup\psi).$$}
\end{enumerate}
\end{lemma}
\begin{proof}
The proof is obtained by using induction on the number of squares of the form $Q_s(sz)$ that admit a non-empty intersection with $\psi$. If $\psi=\es$, then the conditions of Lemma~\ref{usLem1} are satisfied and we can use the description of $\hco(\varphi_1\cup\varphi_2,\cdot)$ given there. In order to verify items $1$ and $2$ suppose that $x\in\varphi_1\cup\varphi_2$ is contained in the symmetric difference of $\enviIn_z(\varphi_1\cup\varphi_2)$ and $\enviIn_z(\varphi_{1})$ or in the symmetric difference of $\enviOut_z(\varphi_1\cup\varphi_2)$ and $\enviOut_z(\varphi_{1})$. By the representation of $\hco$ in Lemma~\ref{usLem1}, this can only happen if $\hge(\varphi_{1},x)\in Q_s(sz_0)$. But in this case, the definition of the event $A_s^\p$ guarantees that also $\hge(\varphi_{1}\cup\varphi_2,x)\in Q_s(sz_0)$, \sot $x$ cannot lie in either of the symmetric differences. Concerning item $3$ if $\hco(\varphi_{1}\cup\varphi_2,x)\not\in\varphi_2$, then we are in the third case of the representation in Lemma~\ref{usLem1}, and the assertion follows.

Next, we proceed with the induction step and decompose $\psi$ as $\psi=\psi^{(1)}\cup\psi^{(2)}$, where $\es\ne\psi^{(1)}\subset Q^\M_s(sz^\p)$ and $\psi^{(2)}\subset\R^{2,\M}\setminus Q^\M_s(sz^\p)$ for some $z^\p\in \Z^2$. Applying the induction hypothesis with $\psi^{(2)}$ instead of $\psi$ and $\psi^{(1)}$ instead of $\varphi_2$, we see that that $\enviIn_z(\varphi_{1}\cup \psi)=\enviIn_z(\varphi_{1})$ and $\enviOut_z( \varphi_{1}\cup\psi)=\enviOut_z (\varphi_{1})$ \fa $z\in\Z^2$. Hence, we may again use the description of $\hco(\varphi_1\cup\varphi_2,\cdot)$ provided in Lemma~\ref{usLem1}. Items $1$-$3$ can now be checked using similar arguments as in the case $\psi=\es$, but for the convenience of the reader, we give some details. Concerning items $1$ and $2$ suppose that $x\in\varphi_1\cup\varphi_2\cup\psi$ is contained in the symmetric difference of $\enviIn_z(\varphi_1\cup\varphi_2\cup\psi)$ and $\enviIn_z(\varphi_{1})$ or in the symmetric difference of $\enviOut_z(\varphi_1\cup\varphi_2\cup\psi)$ and $\enviOut_z(\varphi_{1})$. This is only possible if $x\in\varphi_1\cup\psi$. Like in the case $\psi=\es$, we can exclude the option $\hge(\varphi_{1}\cup\psi,x)\in Q_s(sz_0)$. Finally, in the remaining case, we have $\hco(\varphi_1\cup\varphi_2\cup\psi,x)=\hco(\varphi_1\cup\psi,x)$, and we may use the induction hypothesis to conclude that $x$ cannot be contained in either of the symmetric differences. The third item can now be verified using precisely the same argumentation as in the case $\psi=\es$.
\end{proof}
\begin{proof}[Proof of Proposition~\ref{stopCond}]
It just remains to observe that part (a) of condition (US) follows from Lemma~\ref{stabCycle}, whereas part (b) follows from Lemma~\ref{usLem2}.
\end{proof}

\section{Possible extensions}
\label{extSec}
The present section concludes the paper by discussing possible extensions of the sprinkling approach to other Poisson-based directed random geometric graphs of out-degree at most $1$. The aim of the organization of the proof for the absence of forward percolation (Theorem~{\hyperref[part2]{\ref*{mainProp}.\ref*{part1}}}) was to highlight that the arguments depend on the specific model only via three crucial properties: continuity, the shielding condition (SH) and the uniform stopping condition (US). Additionally, to simplify the presentation, we used sometimes that the geometric descendant of $x\in X$ lies on the line segment connecting $\hco(X,x)$ and $\hge(X,\hco(X,x))$, but removing this condition for a specific example should only be a minor issue.

Apart from the anisotropic lilypond model that we have discussed in detail, another example to which the sprinkling technique applies is given by the directed random geometric graph on a homogeneous Poisson point process, where for some fixed $k\ge1$ for each point a descendant is chosen among the $k$ nearest neighbors according to some distribution. In fact, the verification of the crucial conditions for this example is far less involved than in the lilypond setting.

Moreover, it would be interesting to extend the sprinkling technique to further models of lilypond type. The common characteristic of the lilypond model considered in this paper and the lilypond model introduced in~\cite{lilypond6} is an asymmetry in the growth-stopping protocol. When two line segments hit only one of the two ceases its growth. In both models this asymmetry prevents one from using the classical argumentation for proving absence of percolation, which is based on the absence of suitable descending chains. Although the sprinkling approach seems to be sufficiently strong to deal also with the example considered in~\cite{lilypond6},  there are two important differences that make the verification of conditions (SH) and (US) considerably more involved. First, the latter model is isotropic, \sot the shielding property now has to prevent trespassings in all directions simultaneously. Second, it features two-sided growth \sot the sprinkling has to stop line segments entering a square at roughly the same point in time as before in order to ensure that the configuration of the lilypond model outside the square remains largely unchanged. D.~J.~Daley also proposed to investigate a higher-dimensional analog in $\R^d$, where the two-sided line segments are replaced by the intersection of balls with isotropic codimension $1$ hyperplanes. Since the sprinkling approach is a priori not restricted to the planar setting, it would be very interesting to investigate whether it is also applicable for proving the absence of percolation in these higher-dimensional lilypond models.

\subsection*{Acknowledgments}
The author is grateful for the detailed reports by the anonymous referees that helped to substantially improve the quality of earlier versions of the manuscript and, in particular, for correcting an error in condition (SH). The author thanks D.~J.~Daley for proposing the percolation problem concerning the line-segment model and the generalization to percolation in directed graphs with out-degree $1$. The author also thanks S.~Ziesche and G.~Bonnet for interesting discussions and useful remarks on earlier versions of the manuscript. This work has been supported by a research grant from DFG Research Training Group 1100 at Ulm University.

\bibliography{template}
\bibliographystyle{abbrv}
\end{document}